\documentclass[11pt,a4paper,subeqn,reqno]{article}

\usepackage{graphicx}
\usepackage{amssymb}
\usepackage{epstopdf}

\usepackage{latexsym}

\usepackage{amsmath,amsthm}

\usepackage[mathscr]{eucal}
\usepackage[all]{xy}
\usepackage{mathrsfs}
\usepackage{authblk}
\usepackage{hyperref}

\allowdisplaybreaks
\usepackage[text={169mm,245mm},centering]{geometry}

\newtheorem{theorem}{Theorem}
\newtheorem{lemma}[theorem]{Lemma}

\newtheorem{corollary}[theorem]{Corollary}
\newtheorem{proposition}[theorem]{Proposition}

\theoremstyle{definition}
\newtheorem{example}[theorem]{Example}
\newtheorem{definition}[theorem]{Definition}
\newtheorem{definition-lemma}[theorem]{Definition-Lemma}
\newtheorem{definition-theorem}[theorem]{Definition-Theorem}
\newtheorem{assumption}[theorem]{Assumption}

\newtheorem{remark}[theorem]{Remark}
\newtheorem{convention}[theorem]{Convention}

\newtheorem*{g_convention}{Convention}
\newtheorem*{ack}{Acknowledgements}

\newcommand{\ov}{\overline}

\newcommand{\Hom}{\mathrm{Hom}}

\newcommand{\Hoch}{\mathrm{CH}}
\newcommand{\Cycl}{\mathrm{CC}}

\newcommand{\hdot}{\;\raisebox{3.2pt}{\text{\circle*{2.5}}}}

\def\ldb{\mathopen{\{\!\!\{}}
\def\rdb{\mathclose{\}\!\!\}}}
\def\ldbg{\mathopen{\bigl\{\!\!\bigl\{}}
\def\rdbg{\mathclose{\bigr\}\!\!\bigr\}}}

\begin{document}

\title{A Double Poisson Algebra Structure on Fukaya Categories}

\author[1]{Xiaojun Chen\thanks{Email: xjchen@scu.edu.cn}}
\author[2]{Hai-Long Her\thanks{Email: hailongher@126.com}}
\author[3,4]{Shanzhong Sun\thanks{Email: sunsz@cnu.edu.cn}}
\author[1]{Xiangdong Yang\thanks{Email: xiangdongyang2009@gmail.com}}

\renewcommand\Affilfont{\small}

\affil[1]{Department of Mathematics, Sichuan University, Chengdu 610064 P. R. China}
\affil[2]{School of Mathematical Sciences, Nanjing Normal University, Nanjing 210023 P. R. China}
\affil[3]{Department of Mathematics, Capital Normal University, Beijing 100048 P. R. China}
\affil[4]{Beijing Center for Mathematics and Information Interdisciplinary Sciences, Beijing 100048 P. R. China}

\date{}

\maketitle

\begin{abstract}
Let $M$ be an exact symplectic manifold with $c_1(M)=0$. Denote by $\mathrm{Fuk}(M)$ the Fukaya
category of $M$. We show that the dual space of the bar construction
of $\mathrm{Fuk}(M)$
has a differential graded noncommutative Poisson structure. As a corollary we get a Lie algebra structure on
the cyclic cohomology $\mathrm{HC}^\bullet(\mathrm{Fuk}(M))$,
which is analogous to the ones
discovered by Kontsevich in noncommutative symplectic geometry and by Chas and Sullivan in string topology.
\end{abstract}

\setcounter{tocdepth}{2}\tableofcontents

\section{Introduction}
\renewcommand{\thetheorem}{\Alph{theorem}}

In this paper we construct a noncommutative Poisson structure on the Fukaya category of
an exact symplectic manifold with vanishing first Chern class. Our motivation comes from the noncommutative
symplectic geometry (\cite{Ko1,Ko2,Gin,BL,CBEG}),
noncommutative Poisson geometry (\cite{VdB,CB,BCER}) and
string topology (\cite{CS1,CS2}).
Let us start with some backgrounds.

Roughly speaking, Fukaya category is an algebraic structure arising in the study of symplectic manifolds, where
the objects are Lagrangian submanifolds and the morphisms
are Lagrangian intersection Floer cochain complexes. As observed by Fukaya (\cite{Fuk93}),
the composition of two morphisms
is not associative, but associative up to homotopy. There are homotopy of homotopies,
and homotopy of homotopies of homotopies, etc.,
forming an A$_\infty$ category, a categorical
generalization of Stasheff's A$_\infty$ algebra.

Ever since its first appearance,
Fukaya category has been a fast
developing topic, and is active in, to name a few, symplectic geometry,
homological and homotopical algebra, noncommutative geometry and mathematical physics.
It is one of the noncommutative {\it symplectic} spaces in Kontsevich's homological mirror symmetry program.
In fact, Kontsevich (\cite{HMS}) and also Costello (\cite{Cos07})
conjectured that the Fukaya category of a Calabi-Yau manifold
is a Calabi-Yau A$_\infty$ category, which
means there is a non-degenerate symmetric bilinear pairing of degree $d$ on the morphism spaces
$$
\langle-,-\rangle: \mathrm{Hom}(A,B)\otimes\mathrm{Hom}(B,A)\to k,
$$
for any objects $A$ and $B$, where $k$ is the ground field
of characteristic zero, such that it is cyclically invariant
\begin{equation}\label{CY_cond}
\langle m_n(a_0,\cdots, a_{n-1}), a_n\rangle=(-1)^{n+|a_0|(|a_1|+\cdots+
|a_n|)}\langle m_n(a_1,\cdots, a_n),a_0\rangle.\tag{$\ast$}
\end{equation}
And the famous homological mirror symmetry conjecture of Kontsevich says that,
the (derived category of the) Fukaya category of a Calabi-Yau manifold should be equivalent,
as Calabi-Yau categories, to the
(derived category of) coherent sheaves of its mirror, and vice versa.

In general, it is very difficult to obtain a non-degenerate pairing on a Fukaya category; some
partial results can be found in Fukaya \cite{Fuk10}.
On the other hand, being Calabi-Yau is very important for Fukaya categories, as they would then have
very nice algebraic and geometric properties (see, for example, Kontsevich-Soibelman \cite{KS06}
and Costello \cite{Cos07}).

One nice property of a Calabi-Yau category is a Lie algebra structure on its cyclic cohomology (as shall be
recalled in later sections),
which is nowadays also called the {\it Kontsevich bracket}, and has found many applications in
noncommutative symplectic/Poisson geometry, representation theory of quiver algebras and Calabi-Yau algebras.
Since this Lie algebra is a main motivation of our study, we would like to say some more words about it.

In two very influential papers \cite{Ko1,Ko2}, Kontsevich first raised his theory of noncommutative symplectic
geometry. In particular, he showed that for a noncommutative symplectic space, the noncommutative 0-forms possess
a Lie algebra structure, whose homology is intimately related to the homology of some corresponding moduli space.
His result was later further studied and developed by Ginzburg in \cite{Gin} and Bocklandt-Le Bruyn in \cite{BL}.
These authors proved that the closed path of a doubled quiver has a Lie algebra structure (the Kontsevich bracket),
which is naturally mapped to the Lie algebra of Hamiltonian functions on the corresponding quiver varieties.
Kontsevich's Lie algebra is first considered by Van den Bergh in \cite{VdB} from the
noncommutative Poisson geometry point of view. The relationship between noncommutative symplectic
and noncommutative Poisson structures is also discussed in \cite[Appendix]{VdB}.

In fact, what Van den Bergh introduced is, for a general associative algebra $A$,
the notion of a \textit{double Poisson bracket}.
If the algebra $A$ possesses a double Poisson bracket, then he showed that the commutator quotient space
$A_\natural:=A/[A,A]$ has a Lie algebra structure, where Kontsevich's Lie algebra is a special case when $A$ is
the path algebra of a doubled quiver.
It turns out that Van den Bergh's double Poisson algebra is a very important case of Crawley-Boevey's
noncommutative Poisson structures (\cite{CB}).
The study of Crawley-Boevey was motivated by his trying to find the \textit{weakest}
condition for an associative algebra $A$ such that the moduli space of representations (representation scheme) of $A$ admits
a Poisson structure. If such a condition is fulfilled, we say $A$ possesses a \textit{noncommutative Poisson structure}.
This idea fits very well to a guiding principle proposed by Kontsevich and Rosenberg (\cite{KR}), namely,
for a noncommutative space, any meaningful noncommutative geometric
structure (such as noncommutative symplectic and Poisson)
should induce its classical counterpart on its moduli space of representations.

Now, let us go back to Fukaya categories.
As we have said, it is in general very difficult to prove that a Fukaya category is indeed a Calabi-Yau category.
Nevertheless, we found that Kontsevich's Lie algebra does not {\it a priori} assume the existence
of a non-degenerate pairing, but cyclic invariance (in an appropriate sense) is essential.
This is exactly the case of Fukaya categories, where the counting of the pseudo-holomorphic disks is
cyclically invariant.
That is to say, there is a natural Lie algebra structure on the cyclic cohomology of a Fukaya category,
and such a Lie algebra is
a consequence of the noncommutative Poisson structure (in the sense of Van den Bergh)
on the Fukaya category, when viewing it as a noncommutative space.
The following is our main theorem:

\begin{theorem}[Theorem \ref{main_thm2}]\label{main_thm}
Let $M$ be an exact symplectic $2d$-manifold with $c_1(M)=0$
and possibly with contact type boundary. Denote by $\mathrm{Fuk}(M)$ the Fukaya category
of $M$. Then the dual space of the bar construction of
$\mathrm{Fuk}(M)$
has a degree $2-d$ differential graded
double Poisson algebra structure in the sense of Van den Bergh.
\end{theorem}

As a corollary (Corollary \ref{main_cor}), the cyclic cohomology of the Fukaya
category of an exact symplectic manifold with vanishing
first Chern class has a degree $2-d$ graded Lie algebra structure.

The rest of the paper is devoted to the proof of Theorem \ref{main_thm}.
It is organized as follows:
In Section \ref{Sect_AinftyDPA} we first recall the definition of A$_\infty$ categories
and their Hochschild and cyclic (co)homologies, and
then construct a double Poisson bracket on a class of A$_\infty$ categories;
in Section \ref{Sect_Fuk} we first briefly recall the construction of Fukaya categories
and then prove Theorem \ref{main_thm}; after that,
we discuss some relations of the main result to string topology, a theory developed by Chas and
Sullivan (\cite{CS1,CS2}); finally, we give the detailed proof of Lemma \ref{cyclic_sum} in Appendix A.

\begin{g_convention}
Throughout the paper, we fix a ground field $k$
of characteristic zero. All vector spaces, their morphisms and tensor products are
assumed to be over $k$.
\end{g_convention}

\begin{ack}We would like to thank Song Yang for helpful communications,
and the anonymous referee for carefully reading the paper and
pointing out imprecisions and errors in the earlier draft.
All authors are partially supported by NSFC No. 11271269.
H.-L. Her and S. Sun are also partially supported by NSFC No. 10901084
and No. 11131004, respectively.
\end{ack}

\setcounter{theorem}{0}
\renewcommand{\thetheorem}{\arabic{theorem}}

\section{A$\sb\infty$ categories and the double Poisson bracket}\label{Sect_AinftyDPA}

In this section we first collect some necessary concepts,
such as A$_\infty$ categories and their Hochschild and cyclic (co)homologies,
and then construct, for a class of A$_\infty$ categories, a double Poisson bracket
on the dual space of their bar construction.

\subsection{A$_\infty$ categories and their homologies}

\begin{definition}[A$_\infty$ category; {\it cf.} \cite{Fuk93,Seidel}]\label{def_Ainfty}
An \textit{A$_\infty$ category} $\mathcal A$ over $k$ consists of a set of objects $\mathrm{Ob}(\mathcal A)$,
a graded $k$-vector space
$\Hom(A_1,A_2)$ for each pair of objects $A_1,A_2\in\mathrm{Ob}(\mathcal A)$, and a sequence of
multilinear maps:
$$m_n: \Hom(A_{n},A_{n+1})
\otimes\cdots\otimes \Hom(A_2, A_3)\otimes \Hom(A_1,A_2)\to \Hom(A_1,A_{n+1}), $$
with degree $|m_n|=2-n$, for $n=1,2,\cdots,$ satisfying the following A$_\infty$ relations:
\begin{equation}\label{higher_htpy}
\sum_{p=1}^{n}\sum_{k=1}^{n-p+1}(-1)^{\mu_{p}} m_{n-k+1}(a_n,\cdots,a_{p+k},
m_{k}(a_{p+k-1},\cdots,a_{p}),a_{p-1},\cdots,a_1)=0,
\end{equation}
where $a_i\in\Hom(A_{i},A_{i+1})$, for $i=1,2,\cdots,n$, and
$\mu_{p}=\sum\limits_{r=1}^{p-1}|a_r|-(p-1)$.
\end{definition}

If all $m_i$ vanish except $m_2$,
then by letting
$a_2\circ a_1:=(-1)^{|a_1|}m_2(a_2,a_1)$
one obtains the usual small not-necessarily-unital graded linear category. If
all $m_i$ vanish except $m_1$ and $m_2$, then
one gets the usual small not-necessarily-unital {\it differential graded} (DG for short) category, with the differential
$d(a):=(-1)^{|a|}m_1(a)$.
If an A$_\infty$ category has only one object, say $A$, then $\Hom(A,A)$ is an
\textit{A$_\infty$ algebra}; and if furthermore, all $m_i$ vanish except $m_1$ and $m_2$,
then $\mathcal A $ is the usual not-necessarily-unital DG algebra with product
$a_2\cdot a_1:=(-1)^{|a_1|}m_2(a_2,a_1)$.
Also, for an A$_\infty$ category $\mathcal A$, since $m_1^2=0$ one may obtain the cohomology level
not-necessarily-unital category $\mathrm H(\mathcal A)$,
where the objects remain the same, while the morphisms between two objects, say $A, B$,
are the $m_1$-cohomology
$\mathrm H_\bullet(\mathrm{Hom}(A,B),m_1)$.

\begin{convention}[The signs]
The sign in equation (\ref{higher_htpy}) is given as follows.
First, for a graded vector space $V$, let $\ov V$ be
the de-suspension of $V$, that is, $(\ov V)_i=V_{i+1}$. Let $\Sigma: V\to \ov V$
be the identity map which maps $v$ to $\ov v$, and let
$$
\begin{array}{cccl}
\Sigma^{\otimes n}:&V\otimes\cdots\otimes V& \longrightarrow&\ov V\otimes\cdots\otimes\ov V\\
&v_1\otimes\cdots\otimes v_n&\longmapsto&(-1)^{(n-1)|v_n|+(n-2)|v_{n-1}|+\cdots+|v_{2}|}\ov v_1\otimes\cdots\otimes \ov v_n
\end{array}
$$
be the $n$-fold tensor of $\Sigma$.
Let $\ov m_n:(\ov V)^{\otimes n}\to \ov V$ be the degree 1 map such that the following diagram
\begin{equation}\label{signrule}
\xymatrixcolsep{5pc}
\xymatrix{V\otimes \cdots\otimes V\ar[d]^{\Sigma^{\otimes n}}
\ar[r]^-{m_n}&V\ar[d]^\Sigma\\
\ov V\otimes\cdots\otimes\ov V\ar[r]^-{\ov m_n}&\ov V
}\end{equation}
commutes. Then equation (\ref{higher_htpy}) is nothing but
\begin{equation}\label{htpyrel}
\sum_{p=1}^{n}\sum_{k=1}^{n-p+1}(-1)^{|\ov a_1|+\cdots+|\ov a_{p-1}|}\ov m_{n-k+1}
(\ov a_n,\cdots,\ov a_{p+k}, \ov m_k(\ov a_{p+k-1},\cdots, \ov a_{p}),
\ov a_{p-1},\cdots,\ov a_1)=0.
\end{equation}
The sign that appears in equation (\ref{htpyrel}) follows from the usual Koszul sign rule.
Namely, the canonical isomorphism
$V\otimes W\stackrel{\cong}{\to}W\otimes V$
is given by
$a\otimes b\mapsto (-1)^{|a||b|}b\otimes a$.
One then obtains equation (\ref{higher_htpy}) by converting equation (\ref{htpyrel}) via diagram (\ref{signrule}).
In the following all signs are assigned in this way.
\end{convention}

There is an alternate description of
the A$_\infty$ structure on $\mathcal A $ given as follows:
Let
$
B(\mathcal A ):=\oplus_{n\ge 0}B(\mathcal A)_n
$ (here $n$ is called the {\it weight}),
where
\begin{eqnarray*}
B(\mathcal A)_0&=&k,\\
B(\mathcal A)_n&=&\bigoplus_{A_1,\cdots, A_{n+1}\in\mathrm{Ob}(\mathcal A)}
\overline{\Hom}(A_{n},A_{n+1})
\otimes\cdots\otimes \overline{\Hom}(A_2, A_3)\otimes
\overline{\Hom}(A_1,A_2),\quad n\ge 1.
\end{eqnarray*}
$B(\mathcal A )$ has a natural co-unital, co-augmented
{\it coalgebra} structure, where the co-product is given by
\begin{eqnarray*}
\Delta(\ov a_n,\cdots,\ov a_2,\ov a_1)&=&1\otimes(\ov a_n,\cdots,\ov a_1)+(\ov a_n,\cdots,\ov a_1)\otimes 1\\
&+&
\sum_{i=1}^{n-1}(\ov a_n,\cdots, \ov a_{i+1})\otimes
(\ov a_{i},\cdots, \ov a_1).
\end{eqnarray*}
Grade the elements in $B(\mathcal A)$ by the sum of the gradings of their components,
then $B(\mathcal A)$ is in fact a graded coalgebra, and
equation (\ref{htpyrel}) is equivalent to saying that $\ov m:=\ov m_1+\ov m_2+\cdots$ is nothing but a degree one
co-differential on $B(\mathcal A )$.
The pair $\left(B(\mathcal A ),\ov m\right)$ is called the {\it bar construction} of $\mathcal A$.
In the following we will also use $\tilde B(\mathcal A):=\oplus_{n\ge 1}B(\mathcal A)_n$,
which is called the {\it reduced bar construction} of $\mathcal A$.
$\tilde B(\mathcal A)$ is a DG coalgebra without co-unit, where
the co-product (called the {\it reduced co-product})
is given by $\tilde \Delta (\ov a_n,\cdots,\ov a_1)=\sum_{i=1}^{n-1}(\ov a_n,\cdots, \ov a_{i+1})\otimes
(\ov a_{i},\cdots, \ov a_1)$.

We next recall the definition of Hochschild and cyclic homology of A$_\infty$ algebras/categories,
which are a combination of the ones of, for example, Getzler-Jones \cite[\S3]{GJ} and Kontsevich-Soibelman
\cite[\S7.2.4]{KS06} for A$_\infty$ algebras
as well as Keller \cite[\S5]{Keller2}, \cite[\S1.3]{Keller} and Costello \cite[\S7.4]{Cos07} for DG categories.

\begin{definition}[Hochschild homology]\label{def_hoch}
Let $\mathcal A$ be an A$_\infty$ category as above. The \textit{Hochschild
chain complex} $\Hoch_\bullet(\mathcal A)$ of $\mathcal A$ is the chain
complex whose underlying vector space is
\begin{equation}\label{defHoch}
\bigoplus_{n=0}^{\infty}\bigoplus_{A_1,A_2,\cdots,A_{n+1}\in\mathrm{Ob}(\mathcal A)}
\ov{\Hom}(A_{n+1},A_1)\otimes\ov{\Hom}(A_{n},A_{n+1})\cdots\otimes\ov{\Hom}(A_{2}, A_3)\otimes\ov{\Hom}(A_1,A_2)
\end{equation}
with differential $b=b'+b''$, where
\begin{eqnarray*}b'(\ov a_{n+1},\cdots,\ov a_2,\ov a_1)&=&
\sum_{k=1}^{n+1}\sum_{i=1}^{n-k+2}(-1)^{\varepsilon_k}(\ov a_{n+1},\cdots, \ov a_{k+i}, \ov m_i(\ov a_{k+i-1},
\cdots,\ov a_k),\ov a_{k-1},\cdots, \ov a_1),\\
b''(\ov a_{n+1},\cdots,\ov a_2,\ov a_1)&=&\sum_{j=0}^{n-1}\sum_{i=1}^{n-j}
(-1)^{\nu_{ij}}(\ov m_{i+j+1}(\ov a_{i},\cdots,\ov a_{1}, \ov a_{n+1},\cdots,\ov a_{n-j+1}),
\ov a_{n-j}\cdots,\ov a_{i+1}),
\end{eqnarray*}
where $\varepsilon_k=|\ov a_1|+\cdots+|\ov a_{k-1}|$ and
$\nu_{ij}=(|\ov a_1|+\cdots+|\ov a_i|)(|\ov a_{i+1}|+\cdots+|\ov a_{n+1}|)+(|\ov a_{n-j}|+\cdots+|\ov a_{i+1}|)$.
The associated homology is called the \textit{Hochschild homology} of $\mathcal A$,
and is denoted by $\mathrm{HH}_\bullet(\mathcal A)$.
\end{definition}

In the above definition, if $m_i=0$ for $i\ge 3$, that is, $\mathcal A$ is a small DG category,
then
$$b(\ov a_{n+1},\cdots, \ov a_2, \ov a_1)=(b'+b'')(\ov a_{n+1},\cdots, \ov a_2, \ov a_1)$$
becomes
\begin{eqnarray*}
b'(\ov a_{n+1},\cdots,\ov a_2,\ov a_1)&=&
\sum_{k=1}^{n+1}(-1)^{\varepsilon_k}(\ov a_{n+1},\cdots, \ov a_{k+1}, \ov m_1(\ov a_k),\ov a_{k-1},\cdots, \ov a_1)\\
&&+\sum_{k=1}^n(-1)^{\varepsilon_k}(\ov a_{n+1},\cdots,\ov m_2(\ov a_{k+1}, \ov a_k),\ov a_{k-1},\cdots,\ov a_1),\\
b''(\ov a_{n+1},\cdots,\ov a_2,\ov a_1)&=&
(-1)^{\nu_{1,0}}(\ov m_{2}(\ov a_{1},\ov a_{n+1}),
\ov a_{n}\cdots,\ov a_{2}),
\end{eqnarray*}
which agrees with the one
introduced by, for example, Keller \cite[\S1.3]{Keller}.

One may also define the {\it Hochschild cohomology} of an A$_\infty$ category (\textit{cf.}
Kontsevich-Soibelman \cite[\S7.1]{KS06} and Seidel \cite[\S1f]{Seidel}), which we will not use in this paper.
Instead, in the following
we are more concerned with the dual complex of (\ref{defHoch}), which is
$$
\prod_{n=0}^{\infty}\prod_{A_1, \cdots,A_{n+1}\in\mathrm{Ob}(\mathcal A)}\mathrm{Hom}\Big(
\ov{\Hom}(A_{n+1},A_1)\otimes\ov{\Hom}(A_{n},A_{n+1})\otimes
\cdots \otimes\ov{\Hom}(A_1,A_2), k\Big)
$$
with the induced dual differential of $b$.
Such a complex was originally used by Connes to define the cyclic cohomology (\textit{cf.} Loday \cite[\S2.4]{Loday}),
and is called
the {\it dual Hochschild chain complex} of $\mathcal A$,
and is denoted by $(\mathrm{CH}^\bullet(\mathcal A),b^*)$.

Now, let $t_0: \ov{\Hom}(A, A)\to\ov{\Hom}(A, A)$ be the identity map, and
\begin{multline*}t_{n}: \ov{\Hom}(A_{n+1},A_1)\otimes\cdots\otimes\ov{\Hom}
(A_2, A_3)\otimes\ov{\Hom}(A_1,A_2)\\
\longrightarrow
\ov{\Hom}(A_1,A_2)\otimes\ov{\Hom}(A_{n+1},A_1)\otimes\cdots\otimes\ov{\Hom}(A_{2},A_{3}),
\end{multline*}
for $n=1,2, \cdots$, be the multilinear cyclic operator
\begin{eqnarray}\label{def_t}
t_n(\ov a_{n+1}, \ov a_{n},\cdots,\ov a_{2},\ov a_{1})&:=&
(-1)^{|\ov a_{1}|(|\ov a_{n+1}|+\cdots+|\ov a_{2}|)}(\ov a_1,\ov a_{n+1},\ov a_n,\cdots,\ov a_2)\nonumber\\
&=&(-1)^{|\ov a_{1}|(\sum_{l=1}^{n+1}|\ov a_{l}|-|\ov a_{1}|)}(\ov a_1,\ov a_{n+1},\ov a_n,\cdots,\ov a_2).
\end{eqnarray}
Let $N_n=1+t_n+t_n^2+\cdots+t_n^{n}$.
Extend $t_n$ and $N_n$ to other elements in $\Hoch_\bullet(\mathcal A)$ trivially, and let
$T=t_0+t_1+t_2+\cdots$ and $N=N_0+N_1+N_2+\cdots$.

\begin{lemma}\label{cyclic_sum}
Let $\mathcal A $ be an A$_\infty$ category, and let $T$ and $N$ be as above.
We have the following commutative diagram:
\begin{equation}\label{commdiag_cmpx}
\xymatrixcolsep{3.5pc}
\xymatrix{
\mathrm{CH}_n(\mathcal A)\ar[r]^-{N}\ar[d]^-{b}&
\mathrm{CH}_n(\mathcal A)\ar[d]^-{b'}\ar[r]^-{1-T}&
\mathrm{CH}_n(\mathcal A)\ar[d]^b\\
\displaystyle\bigoplus_{m=0}^{n}\mathrm{CH}_m(\mathcal A)\ar[r]^-{N}&
\displaystyle\bigoplus_{m=0}^{n}
\mathrm{CH}_m(\mathcal A)\ar[r]^-{1-T}&
\displaystyle\bigoplus_{m=0}^{n} \mathrm{CH}_m(\mathcal A).
}
\end{equation}
\end{lemma}

\begin{proof}
This is the A$_\infty$ version of cyclic bicomplex (\textit{cf.} \cite[\S2.1.2]{Loday}).
Since the computation involves the higher A$_\infty$ operators, which
seems to have not appeared in literature before, we give a proof in the appendix.
\end{proof}

\begin{definition}[Cyclic homology]\label{def_cycl}
Suppose $\mathcal A$ is an A$_\infty$ category.
The cokernel
$\Hoch_\bullet(\mathcal A)/(1-T)$ of $1-T$ forms a chain complex with the induced differential from the
Hochschild $b$-complex (still denoted by $b$).
Such chain complex is denoted by $\Cycl_\bullet(\mathcal A)$, and is called the \textit{Connes cyclic
complex} of $\mathcal A$. Its homology is called the \textit{cyclic homology} of $\mathcal A$, and is denoted by
$\mathrm{HC}_\bullet(\mathcal A)$.

The \textit{cyclic cochain complex} of $\mathcal A$ is the cyclically invariant sub-complex of the
dual Hochschild chain complex
$\Hoch^\bullet(\mathcal A )$, and is denoted by $\Cycl^\bullet(\mathcal A)$. Namely, suppose
$f\in\Hoch^\bullet(\mathcal A )$, then $f\in\Cycl^\bullet(\mathcal A)$ if and only if for all
$\alpha\in\Hoch_\bullet(\mathcal A)$, $f(\alpha)=f(T(\alpha))$. The corresponding cohomology
is called the \textit{cyclic cohomology} of $\mathcal A $ and is denoted by $\mathrm{HC}^\bullet(\mathcal A )$.
\end{definition}

\begin{remark}The definition of the cyclic homology and cohomology of A$_\infty$ algebras can
be found in Getzler-Jones \cite{GJ} and Penkava-Schwarz \cite{PS} respectively.
In literature Keller first defined the Hochschild homology of (small) DG categories (\cite{Keller2,Keller}),
which are given by formulas in Definitions \ref{def_hoch} with all $m_n$ ($n\ge 3$) vanishing.
His definition of cyclic homology for DG categories is slightly different but equivalent to Definition \ref{def_cycl}
in the case the DG category has a unit (see \cite[\S2.2-4 and \S5.4]{Keller2}).

In the above definition of A$_\infty$ categories, we did not require the category to have a unit. Indeed in symplectic geometry
the Fukaya categories may not have a unit, however, they are ``cohomologically unital", which means
the homology of a Fukaya category is a graded category with unit (for a proof of this statement
see Seidel \cite[\S9j]{Seidel}).
It is known from homological algebra ({\it cf.} \cite[Corollary 2.14]{Seidel}),
as will be recalled later, that any homologically unital
A$_\infty$ category, say $\mathcal A$, is canonically
quasi-isomorphic to a (unital) DG category, say $\mathcal I\mathcal A$.
Under this quasi-isomorphism the Hochschild and cyclic
complexes of $\mathcal A$ are mapped to the ones of $\mathcal I\mathcal A$ respectively.
\end{remark}

From the definitions,
we see that $\tilde B(\mathcal A )$ contains $\mathrm{CH}_\bullet(\mathcal A )$ as a subspace
and $\overline m$ on the former restricts to $b'$ on the latter.
Also, let
$$\tilde B(\mathcal A )^{\vee}=\mathrm{Hom}(\tilde B(\mathcal A ), k)=
\prod_{n=1}^{\infty}\prod_{A_1, \cdots,A_{n+1}\in\mathrm{Ob}(\mathcal A)}\mathrm{Hom}\Big(
\ov{\Hom}(A_{n},A_{n+1})\otimes
\cdots \otimes\ov{\Hom}(A_1,A_2), k\Big)
$$
with the dual differential $\ov m^{\vee}$,
equipped with the adic topology
indexed by the natural numbers with the usual order
and the subset of objects in $\mathcal A$ under inclusion.
$(\tilde B(\mathcal A )^{\vee}, \ov m^{\vee})$
has a natural non-unital DG algebra structure, where
the product is given by
$$(f\hdot g)(\overline a_n,\cdots,\ov a_1):=\mu\circ (f\otimes g)\circ\tilde \Delta(\ov a_n,\cdots,\ov a_1),$$
with $\mu$ the product on $k$.
We have the following proposition which is originally due to Quillen in
the case of a differential graded algebra:

\begin{proposition}[Quillen \cite{Quillen89} Lemma 1.2]\label{Prop_Quillen}
Suppose $\mathcal A $ is an A$_\infty$ category.
Then
\begin{equation}\label{com_quo}
\Cycl_\bullet(\mathcal A )\simeq \left( \tilde B(\mathcal A )\right)^\natural,\quad
\Cycl^\bullet(\mathcal A )\simeq \left(\tilde B(\mathcal A )^{\vee}\right)_\natural,
\end{equation}
where $(-)^\natural$ means the co-commutator subspace and $(-)_\natural$
is the topological commutator quotient space (i.e. the quotient by
the closure of the commutators
under the adic topology).
\end{proposition}

\begin{proof}
Recall that for a DG coalgebra $C$,
the co-commutator subspace $C^\natural$ is the subspace
$$C^\natural:=\mathrm{Ker}\{\tilde \Delta-\sigma\circ\tilde \Delta:C\to C\otimes C\},$$
where $\tilde \Delta$ is the reduced co-product and $\sigma$ is the switching operator $\sigma(a, b)=(-1)^{|a|\cdot|b|}(b, a)$.
Thus if
$$x\in \ov{\Hom}(A_{n},A_{n+1})
\otimes\cdots\otimes\ov{\Hom}(A_2, A_3)\otimes\ov{\Hom}(A_1,A_2)\subset\tilde B(\mathcal A )$$
lies in
$\left(\tilde B(\mathcal A )\right)^{\natural}$, then we have that $A_{n+1}=A_1$ otherwise
$\tilde \Delta(x)-\sigma\circ\tilde \Delta(x)$ cannot be eliminated. This means
$x\in (\mathrm{CH}_\bullet(\mathcal A ), b')$.
In this case, following Quillen \cite[Lemma 1.2]{Quillen89},
if we let $p_{i,j}(v)$ be the component of $v\in\tilde B(\mathcal A)\otimes \tilde B(\mathcal A)$ of
bi-weight $(i,j)$,
then
\begin{eqnarray*}
&&p_{i, n-i}\circ \sigma\circ\tilde \Delta(a_1,\cdots, a_n)\\
&=&(-1)^{(|a_1|+\cdots+|a_{n-i}|)(|a_{n-i+1}|+\cdots+|a_n|)}(a_{n-i+1},\cdots, a_n)\otimes(a_1,\cdots, a_{n-i})\\
&=&p_{i, n-i}\circ \tilde \Delta \circ T^i(a_1,\cdots, a_n),
\end{eqnarray*}
which means $p_{i, n-i}\circ (\tilde \Delta-\sigma\circ\tilde \Delta)(x)=p_{i, n-i}\circ\tilde \Delta\circ(1-T^i)(x)$.
Thus if $x\in\mathrm{Ker}(1-T)$,
then
$(\tilde \Delta-\sigma\circ\tilde \Delta)(x)=0$,
and conversely,
if $(\tilde \Delta-\sigma\circ\tilde \Delta)(x)=0$,
then by taking $i=1$ we immediately obtain $x\in\mathrm{Ker}(1-T)$.
This means that
$$\left(\tilde B(A)\right)^{\natural}\cong \mathrm{Ker}\{1-T\}.$$
Since $\mathrm{Coker}(1-T)\cong \mathrm{Ker}(1-T)$
and $b'N=Nb$, we obtain
$$\mathrm{CC}_\bullet(\mathcal A )=\mathrm{Coker}(1-T)
\cong\mathrm{Ker}(1-T)=\left(\tilde B(\mathcal A )\right)^{\natural}$$
as complexes.
Similarly we also have $\mathrm{CC}^\bullet(\mathcal A )\cong (\tilde B(\mathcal A )^\vee)_{\natural}$.
This completes the proof.
\end{proof}

\subsection{Construction of the double bracket}

We first recall Van den Bergh's definition of double Poisson algebra (\cite{VdB}).

\begin{definition}[Double Poisson algebra]\label{def_DPB}
Suppose $A$ is a graded associative algebra over $k$.
A \textit{double bracket} of degree $d$ on $A$ is a bilinear map
$$
\ldb-,-\rdb:A\otimes A\to A\otimes A
$$
which is a derivation of degree $d$ (for the outer $A$-bimodule structure on $A\otimes A$) in its second argument
and satisfies
\begin{equation}\label{db}
\ldb a,b\rdb=-{(-1)}^{(|a|+d)(|b|+d)}\ldb b,a\rdb^\circ,
\end{equation}
where $(u\otimes v)^\circ ={(-1)}^{|u||v|}v\otimes u$.

Suppose that $\ldb-,-\rdb$ is a double bracket of degree $d$ on $A$.
For $a,b_1,...,b_n$ homogeneous in $A$, let
$$\ldb a, b_1 \otimes \ldots \otimes b_n\rdb_L \,:=\, \ldb a,b_1 \rdb \otimes b_2 \otimes \ldots \otimes b_n \,\text{,}$$
and let
$$\sigma_s(b_1 \otimes \ldots \otimes b_n)\,:=\, {(-1)}^t b_{s^{-1}(1)} \otimes \ldots \otimes b_{s^{-1}(n)}\,$$
where
$ t:=\sum_{i<j; s^{-1}(j)<s^{-1}(i)} |b_{s^{-1}(i)}||b_{s^{-1}(j)}|$.
If furthermore $A$ satisfies the following
\textit{double Jacobi identity}
\begin{equation}\label{dJ}
\ldbg a , \ldb b,c \rdb \rdbg_L + (-1)^{(|a|+d)(|b|+|c|)} \sigma_{(123)}\ldbg b,\ldb c,a\rdb \rdbg_L + (-1)^{(|c|+d)(|a|+|b|)} \sigma_{(132)}
\ldbg c,\ldb a,b\rdb \rdbg_L =0,
\end{equation}
then $A$ is called a \textit{double $d$-Poisson algebra}, or a {\it double Poisson algebra of degree $d$}.
A DG algebra $(A,\partial)$ is said to have a DG double $d$-Poisson structure
if $A$ admits a double $d$-Poisson structure which commutes with $\partial$.
\end{definition}

In the following, we focus on the case where $A$ is $\tilde{B}(\mathcal A)^{\vee}$ for an A$_\infty$
category $\mathcal A$. Note that we may consider
the double bracket in the complete sense, that is,
both the domain and the image may be taken to be $A\hat\otimes A$,
where $\hat\otimes$ is the completed tensor product under the adic topology,
and the double Jacobi identity is to hold in this complete sense.

\begin{assumption}\label{assumption_CYAlt}
In the following we consider a class of A$_\infty$ categories satisfying the following conditions:

(1) there exists a positive integer $d$ such that for each pair of objects $A,B\in\mathcal A $,
there is an isomorphism of finite dimensional $k$-vector spaces
\begin{equation}\label{Poincareduality}
\mathrm{Hom}^i(A,B)\cong\mathrm{Hom}^{d-i}(B,A),
\end{equation}
for all $i$, and under this isomorphism, a basis of $\mathrm{Hom}(A,B)$
is map to a basis of $\mathrm{Hom}(B,A)$ (in the following we denote by $p^*$
the image of an element $p$ of the basis under this isomorphism);

(2) shift the gradings of the elements in the morphism space of $\mathcal A$ down by one;
for basis elements
$p_n\in \ov{\mathrm{Hom}}( A_n, A_{n+1}), p_{n-1}\in\ov{\mathrm{Hom}}(A_{n-1},A_n),
\cdots, p_{1}\in\ov{\mathrm{Hom}}(A_1,A_2)$, write
\begin{equation*}
\ov m_n(p_n, p_{n-1},\cdots, p_1)=\sum_{q }
\epsilon(q^*, p_{n},\cdots, p_1)\cdot q,
\end{equation*}
where $q$ runs over the basis of $\ov{\mathrm{Hom}}(A_1, A_{n+1})$, and
$\epsilon(q^*, p_{n},\cdots, p_1)\in k$,
then $\epsilon(q^*, p_{n},\cdots, p_1)$ is cyclically invariant, that is,
\begin{equation}\label{cyclic_inv_counting}
\epsilon(q^*, p_{n},\cdots, p_1)
=(-1)^{|q^*|(|p_1|+\cdots+|p_n|)}
\epsilon(p_{n},\cdots, p_1,q^*).
\end{equation}
\end{assumption}

\begin{convention}
\label{sign_conv}
For some sign issues,
in the following we make the following convention:
for each pair of basis elements $p, p^*$, we have two ordered set $(p,p^*)$ and $(p^*,p)$,
and assign a sign to one of them
by
$$\mathrm{sgn}(p,p^*)=(-1)^{|p|}$$
if $p\ne p^*$ or if $p=p^*$ and $p$ is of odd degree (and in this case
$\mathrm{sgn}(p^*,p)=(-1)^{|p^*|+(|p^*|+1)(|p|+1)}$ by the Koszul sign convention),
and
$$\mathrm{sgn}(p,p^*)=0$$
if $p=p^*$ and $p$ is of even degree.
There is a choice in assigning the signs, but once assigned, they are fixed in the rest.
To get some idea about the sign, let us remind that in
the works of Ginzburg \cite{Gin} and Van den Bergh \cite{VdB}, in order to construct
the Lie/Poisson bracket (respectively, the double Poisson bracket) on the closed path space (respectively,
the path algebra)
of a doubled quiver, one has to equip a {\it symplectic} pairing on the space of edges:
given a quiver $Q$, first double it, that is, to each edge $a$ in $Q$, add one more edge $a^*$
but with arrow reversed; then the symplectic pairing is given by
$\langle a, b^*\rangle=
-\langle b^*, a\rangle=1$ if $b^*=a^*$ and $0$ otherwise. However, if the quiver is already doubled,
then for each pair of such edges, one has to choose which one is $a$ and which one is $a^*$ to
define the symplectic pairing; whichever
is chosen as the original edge will not affect the conclusion. The sign given
above is just a DG version
of theirs.
\end{convention}

\begin{proposition}\label{main_lemma}
Let $\mathcal A $ be an A$_\infty$ category satisfying Assumption \ref{assumption_CYAlt}, and let
$ \tilde{B}(\mathcal A )^{\vee}$
be the dual space of the reduced bar construction of $\mathcal A $.
Define $$\ldb-,-\rdb: \tilde{B}(\mathcal A )^{\vee}
\hat\otimes \tilde{B}(\mathcal A )^{\vee}
\to  \tilde{B}(\mathcal A )^{\vee}\hat\otimes  \tilde{B}(\mathcal A )^{\vee}$$
by the following formula: for homogeneous $f,g\in \tilde{B}(\mathcal A )^{\vee}$,
\begin{eqnarray}
&&\ldb f,g\rdb((a_m,\cdots,a_1),(b_n,\cdots,b_1))\nonumber\\
&=&\sum_{i=1}^{m+1}\sum_{j=1}^{n+1}\sum_{p }(-1)^{\varepsilon_{ij}}\mathrm{sgn}(p,p^*)
f(b_n,\cdots, b_j,p,a_{i-1},\cdots,a_1)g(a_m,\cdots,a_i,p^*,b_{j-1},\cdots,b_1),\;\;\;\;\;\;\label{formula_dpb}
\end{eqnarray}
where
\begin{eqnarray*}
u=(a_m,\cdots, a_1)&\in& \ov{\mathrm{Hom}}(A_{m}, A_{m+1})\otimes\cdots\otimes\ov{\mathrm{Hom}}(A_1,A_2),\\
v=(b_n,\cdots,b_1)&\in&\ov{\mathrm{Hom}}(B_{n}, B_{n+1})\otimes\cdots\otimes\ov{\mathrm{Hom}}(B_1,B_2),
\end{eqnarray*}
$\varepsilon_{ij}=(|a_m|+\cdots+|a_i|+|b_n|+\cdots+|b_j|+|p|)(|a_{i-1}|+\cdots+|a_1|)+(|a_m|+\cdots+|a_i|)(|b_n|+\cdots+|b_j|+|p^*|)$,
and
$p$ runs over the basis of $\ov{\mathrm{Hom}}(A_i,B_j)$.
Then $\ldb-,-\rdb$ defines in the complete sense a DG double Poisson bracket of degree $2-d$ on
$ \tilde{B}(\mathcal A )^{\vee}$.
\end{proposition}

\begin{remark}
In the above proposition,
$\ldb-,-\rdb$ can in fact be extended to be defined on $B(\mathcal A)^{\vee}$, that is,
in the formula (\ref {formula_dpb}),
if $f$ or $g$ is in $k$, then one simply puts $\ldb f,g\rdb=0$.
Also,
the summands in the right-hand side of (\ref{formula_dpb}) for $j=1$ and $n+1$ and arbitrary $i$
are understood as
$$
f(b_n,\cdots, b_1, p, a_{i-1},\cdots, a_1)\cdot g(a_m,\cdots, a_i, p^*)
$$
and
$$
f( p, a_{i-1},\cdots, a_{1},)\cdot g(a_m,\cdots, a_i, p^*, b_n, \cdots, b_1)
$$
respectively, and for $i=1$ or $m+1$ the formulas are similarly given.
\end{remark}

\begin{proof}[Proof of Proposition \ref{main_lemma}]

First note that $\ldb-,-\rdb$
has degree $2-d$: the difference of degrees between
$$|(a_m,\cdots,a_1)|+|(b_n,\cdots,b_1)|\;\mbox{
and}\;
|(b_n,\cdots, b_j,p,a_{i-1},\cdots,a_1)|+|(
a_m,\cdots,a_i,p^*,b_{j-1},\cdots,b_1)|$$
is $|p|+|p^*|$, which is $d-2$ (recall that by our convention, all elements in
the bar construction have shifted their degree down by $1$), and
therefore the double bracket has degree $2-d$.

The double bracket $\ldb-,-\rdb$ is graded skew-symmetric, {\it i.e.} it satisfies equation (\ref{db}):
for any
$f, g\in \tilde{B}(\mathcal A )^{\vee}$, and any
$u=(a_m,\cdots, a_1),v=(b_n,\cdots,b_1)$ in $\tilde{B}(\mathcal A )$,
\begin{eqnarray*}
& & \ldb g, f\rdb^\circ ((a_m,\cdots, a_1),(b_n,\cdots, b_1))\\
&=&(-1)^{|u||v|}\ldb g,f\rdb((b_n,\cdots,b_1),(a_m,\cdots, a_1))\\
&=&\sum_{i,j}\sum_{p^*}(-1)^{|u||v|+\tilde\varepsilon_{ji}}\mathrm{sgn}(p^*,p)
g(a_m,\cdots, a_i,p^*,b_{j-1},\cdots, b_1)  f(b_n,\cdots, b_j,p, a_{i-1},\cdots, a_1)\\
&=&\sum_{i,j}\sum_p(-1)^{|u||v|+\tilde\varepsilon_{ji}+\varepsilon_{ij}+|p||p^*|+1} (-1)^{\varepsilon_{ij}}\mathrm{sgn}(p,p^*)\\
&&\quad\quad\quad\quad\quad\cdot
f(b_n,\cdots, b_j,p, a_{i-1},\cdots, a_1)
g(a_m,\cdots, a_i,p^*,b_{i-1},\cdots, b_1)\\
&=&-(-1)^{(|f|+d)(|g|+d)}\ldb f, g\rdb((a_m,\cdots,a_1),(b_n,\cdots, b_1)).
\end{eqnarray*}
where $\tilde\varepsilon_{ji}$ is similar to $\varepsilon_{ij}$ but with
$a_*$ replaced by $b_*$ and vice versa, and with $p,p^*$ switched.
A direct computation shows
$$(-1)^{|u||v|+\varepsilon_{ij}+\tilde\varepsilon_{ji}+|p||p^*|}=(-1)^{(|f|+d)(|g|+d)}$$
if
$\ldb f, g\rdb((a_m,\cdots,a_1),(b_n,\cdots, b_1))$ is non-zero. Basically the signs
given above follow the Koszul sign rule; the negative sign in the RHS
of the last equality comes from Convention \ref{sign_conv}, namely, whenever
$p$ and $p^*$ appear in an expression simultaneously,
then there is a negative sign added besides the Koszul sign, when their orders are switched.

We now show that $\ldb-,-\rdb$ is a derivation for the second component.
For $f, g, h\in \tilde{B}(\mathcal A )^{\vee}$,
suppose $\ldb f, g\rdb= \alpha'\hat\otimes \alpha''$, $\ldb f, h\rdb=\beta'\hat\otimes\beta''$, then
\begin{eqnarray*}
&&\ldb f,g\cdot h\rdb((a_m,\cdots, a_1),(b_n,\cdots,b_1))\\
&=&\sum_{i,j}\sum_p(-1)^{\varepsilon_{ij}}\mathrm{sgn}(p,p^*) f(b_n,\cdots, b_j,p,a_{i-1},\cdots,a_1) (g h)(a_m,\cdots,a_i,p^*,b_{j-1},\cdots,b_1)\\
&=&\sum_{i,j,k}\sum_p(-1)^{\varepsilon_{ij}+|h|(|a_m|+\cdots+|a_{k+1}|)}\mathrm{sgn}(p,p^*) f(b_n,\cdots, b_j,p,a_{i-1},\cdots,a_1) \\
&&\quad\quad\quad\quad \quad\quad\quad\quad \cdot g(a_m,\cdots,a_{k+1}) h(a_k,\cdots, a_i,p^*,b_{j-1},\cdots, b_{1})\\
&+&\sum_{i,j,k}\sum_p(-1)^{\varepsilon_{ij}+|h|(|a_m|+\cdots+|a_i|+|p^*|+|b_{j-1}|+\cdots+|b_k|) }\mathrm{sgn}(p,p^*) f(b_n,\cdots, b_j,p,a_{i-1},\cdots,a_1)\\
&&\quad\quad\quad\quad \quad\quad\quad\quad \cdot g(a_m,\cdots, a_i, p^*, b_{j-1},\cdots, b_k)
h(b_{k-1},\cdots,b_1)\\
&=&\sum_{i,j,k}\sum_p(-1)^{\varepsilon_{ij}+|h|(|a_m|+\cdots+|a_{k+1}|)}\mathrm{sgn}(p,p^*) g(a_m,\cdots,a_{k+1}) f(b_n,\cdots, b_j,p,a_{i-1},\cdots,a_1) \\
&&\quad\quad\quad\quad \quad\quad\quad\quad \cdot h(a_k,\cdots, a_i,p^*,b_{j-1},\cdots, b_{1})\\
&+&\sum_{i,j,k}\sum_p(-1)^{\varepsilon_{ij}+|h|(|a_m|+\cdots+|a_i|+|p^*|+|b_{j-1}|+\cdots+|b_k|) }\mathrm{sgn}(p,p^*) f(b_n,\cdots, b_j,p,a_{i-1},\cdots,a_1)\\
&&\quad\quad\quad\quad \quad\quad\quad\quad \cdot g(a_m,\cdots, a_i, p^*, b_{j-1},\cdots, b_k)
h(b_{k-1},\cdots,b_1)\\
&=&((-1)^{|\beta'|}g\cdot \beta'\hat\otimes \beta''+\alpha'\hat\otimes \alpha''\cdot h)((a_m,\cdots, a_1),(b_n,\cdots,b_1))\\
&=&((-1)^{|\ldb f,h\rdb'|}g\cdot\ldb f,h\rdb+\ldb f,g\rdb \cdot h)((a_m,\cdots, a_1),(b_n,\cdots,b_1)).\end{eqnarray*}
This means $\ldb-,-\rdb$ is a derivation.

We next show that $\ldb-,-\rdb$ satisfies graded double Jacobi identity \eqref{dJ}, that is, up to Koszul sign,
$$
\ldbg f,\ldb g,h\rdb'\rdbg\hat{\otimes}\ldb g,h\rdb''\pm
\ldb h,f\rdb''\hat{\otimes}\ldbg g,\ldb h,f\rdb'\rdbg \pm \ldbg h,\ldb f,g\rdb'\rdbg''\hat{\otimes}
\ldb f,g\rdb''\hat{\otimes}\ldbg h,\ldb f,g\rdb'\rdbg'=0.
$$
In fact, for $u=(a_m,\cdots,a_1),v=(b_n,\cdots, b_1),w=(c_r,\cdots, c_1)$, we have
\begin{subequations}
\begin{eqnarray}
&&\big(\ldbg f,\ldb g,h\rdb'\rdbg\hat{\otimes}\ldb g,h\rdb''\big)\big((a_m,\cdots,a_1),(b_n,\cdots, b_1),(c_r,\cdots, c_1)\big)\nonumber\\
&=&\sum_{i,j}\sum_p(-1)^{\varepsilon_{ij}+|\ldb g,h\rdb''||w|}\mathrm{sgn}(p,p^*)
f(b_n,\cdots,b_j,p,a_{i-1},\cdots, a_1)\nonumber\\
&&\quad\quad\quad\quad\quad\cdot
\ldb g,h\rdb'(a_m,\cdots,a_i ,p^*,b_{j-1},\cdots, b_1)
\ldb g,h\rdb''(c_r,\cdots,c_1)\nonumber\\
&=&
\sum_{i,j,k,\ell}\sum_{p,q}(-1)^{\varepsilon_{ij}+\nu_{k\ell}
+|\ldb g,h\rdb''||w|}\mathrm{sgn}(p,p^*)\mathrm{sgn}(q,q^*)
f(b_n,\cdots,b_j,p,a_{i-1},\cdots, a_1)\nonumber\\
&&\quad\quad\quad\quad\quad\cdot
g(c_r,\cdots, c_{\ell},q,a_{k-1},\cdots, a_i,p^*,b_{j-1},\cdots, b_1)
h(a_m,\cdots, a_k, q^*, c_{\ell-1},\cdots,c_1)\;\;\;\;\;\label{dJ1a}\\
&+&\sum_{i,j,k,\ell}\sum_{p,q}(-1)^{\varepsilon_{ij}+\eta_{k\ell}
+|\ldb g,h\rdb''||w|}\mathrm{sgn}(p,p^*)\mathrm{sgn}(q,q^*)f(b_n,\cdots,b_j,p,a_{i-1},\cdots, a_1)\nonumber\\
&&\quad\quad\quad\quad\quad\cdot
g(c_r,\cdots, c_{\ell}, q, b_{k-1}, \cdots, b_1)
h(a_m, \cdots, a_i, p^*, b_{j-1}, b_k,q^*,c_{\ell-1},\cdots, c_1),\label{dJ1b}
\end{eqnarray}
\end{subequations}
and
\begin{subequations}
\begin{eqnarray}
&&\big(
\ldb h,f\rdb''\hat{\otimes}\ldbg g,\ldb h,f\rdb'\rdbg \big)
\big((a_m,\cdots,a_1),(b_n,\cdots, b_1),(c_r,\cdots, c_1)\big)\nonumber\\
&=&\sum_{j,k}\sum_p(-1)^{\lambda_{jk}+|u||\ldb g,\ldb,h,f\rdb'\rdb|}\mathrm{sgn}(p,p^*) \ldb h,f\rdb''(a_m,\cdots, a_1)\nonumber\\
&&\quad\quad\quad\quad\quad\cdot g(c_r,\cdots, c_k, p, b_{j-1}, \cdots, b_1)
\ldb h, f\rdb'(b_n,\cdots, b_j,p^*, c_{k-1},\cdots, c_1)\nonumber\\
&=&\sum_{j,k}\sum_p(-1)^{\lambda_{jk}+|u||\ldb g,\ldb,h,f\rdb'\rdb|}\mathrm{sgn}(p,p^*)g(c_r,\cdots, c_k, p, b_{j-1}, \cdots, b_1)\nonumber\\
&&\quad\quad\quad\quad\quad\cdot
\ldb h, f\rdb'(b_n,\cdots, b_j,p^*, c_{k-1},\cdots, c_1) \ldb h,f\rdb''(a_m,\cdots, a_1)\nonumber\\
&=&\sum_{i, j,k,\ell}\sum_{p,q}(-1)^{\lambda_{jk}+\xi_{i\ell}+|u||\ldb g,\ldb,h,f\rdb'\rdb|}\mathrm{sgn}(p,p^*)\mathrm{sgn}(q,q^*)
g(c_r,\cdots, c_k, p, b_{j-1}, \cdots, b_1)\nonumber\\
&&\quad\quad\quad\quad\quad\cdot
h(a_m,\cdots, a_i, q, b_{\ell-1},\cdots, b_j,p^*, c_{k-1},\cdots, c_1)
f(b_n,\cdots,b_{\ell},q^*, a_{i-1},\cdots,a_1)\;\;\;\;\; \label{dJ2a}\\
&+&\sum_{i, j,k,\ell}\sum_{p,q}(-1)^{\lambda_{jk}+\zeta_{i\ell}+|u||\ldb g,\ldb,h,f\rdb'\rdb|}\mathrm{sgn}(p,p^*)\mathrm{sgn}(q,q^*)
g(c_r,\cdots, c_k, p, b_{j-1}, \cdots, b_1)\nonumber\\
&&\quad\quad\quad\quad\quad\cdot
h(a_m,\cdots, a_i, q, c_{\ell-1}, \cdots, c_1)
f(b_n,\cdots,b_j,p^*, c_{k-1},\cdots,c_{\ell},q^*,a_{i-1},\cdots,a_1),\label{dJ2b}
\end{eqnarray}
\end{subequations}
and at last,
\begin{subequations}
\begin{eqnarray}
&&\big( \ldbg h,\ldb f,g\rdb'\rdbg''\hat{\otimes}
\ldb f,g\rdb''\hat{\otimes}\ldbg h,\ldb f,g\rdb'\rdbg'\big)\big((a_m,\cdots,a_1),(b_n,\cdots, b_1),(c_r,\cdots, c_1)\big)\nonumber\\
&=&(-1)^{|u||\ldb f,g\rdb''|+(|u|+|v|)|\ldb h,\ldb f,g\rdb'\rdb'|}
 \ldbg h,\ldb f,g\rdb'\rdbg\big((c_r,\cdots, c_1),(a_m,\cdots,a_1)\big)
\ldb f,g\rdb''(b_n,\cdots,b_1)\nonumber\\
&=&\sum_{i,k}\sum_p (-1)^{\sigma_{ik}+|u||\ldb f,g\rdb''|+(|u|+|v|)|\ldb h,\ldb f,g\rdb'\rdb'|}
\mathrm{sgn}(p,p^*)
h(a_m,\cdots, a_i, p, c_{k-1}, \cdots, c_1)\nonumber\\
&&\quad\quad\quad\quad\quad\cdot
\ldb f,g\rdb'(c_r,\cdots, c_k,p^*, a_{i-1},\cdots, a_1)\cdot
\ldb f,g\rdb''(b_n,\cdots, b_1)\nonumber\\
&=&\sum_{i,j,k,\ell}\sum_{p,q}
 (-1)^{\sigma_{ik}+\tau_{j\ell}+|u||\ldb f,g\rdb''|+(|u|+|v|)|\ldb h,\ldb f,g\rdb'\rdb'|}
\mathrm{sgn}(p,p^*)\mathrm{sgn}(q,q^*)
h(a_m,\cdots, a_i, p, c_{k-1}, \cdots, c_1)\nonumber\\
&&\quad\quad\quad\quad\quad\cdot
f(b_n,\cdots, b_j,q,c_{\ell-1},\cdots, c_k, p^*, a_{i-1},\cdots, a_1)
g(c_r,\cdots,c_{\ell},q^*,b_{j-1},\cdots, b_1)
\label{dJ3a}\\
&+&\sum_{i,j,k,\ell}\sum_{p,q} (-1)^{\sigma_{ik}+\rho_{j\ell}+|u||\ldb f,g\rdb''|+(|u|+|v|)|\ldb h,\ldb f,g\rdb'\rdb'|}
\mathrm{sgn}(p,p^*)\mathrm{sgn}(q,q^*)
h(a_m,\cdots, a_i, p, c_{k-1}, \cdots, c_1)\nonumber\\
&&\quad\quad\quad\quad\quad\cdot
f(b_n,\cdots, b_j,q,a_{\ell-1},\cdots, a_1)
g(c_r,\cdots, c_k, p^*,a_{i-1},\cdots, a_{\ell},q^*,b_{j-1},\cdots, b_1).\label{dJ3b}
\end{eqnarray}
\end{subequations}
In the above expressions, $\nu_{k\ell}, \eta_{k\ell}, \lambda_{jk},\xi_{i\ell},\sigma_{ik},
\tau_{j\ell}$ and $\rho_{j\ell}$
are all defined similarly to $\varepsilon_{ij}$.
After re-arranging the items, one sees that up to Koszul sign,
(\ref{dJ1a}) and (\ref{dJ3b}) cancel with each other (recall that we mentioned above that
$\mathrm{sgn}(p,p^*)$ and $\mathrm{sgn}(p^*,p)$ differ by a sign),
so do (\ref{dJ1b}) and (\ref{dJ2a}), and (\ref{dJ2b}) and (\ref{dJ3a}),
hence the graded double Jacobi identity is verified.

We last prove that the double bracket commutes with the differential $\overline m^{\vee}$. To simplify the notation,
denote $\partial=\overline m^{\vee}$.
We need to show
$\partial\ldb f,g\rdb=\ldb\partial f, g\rdb+(-1)^{|f|}\ldb f,\partial g\rdb$, where the differential $\partial$ acts
on tensor products by derivation. In fact,
\begin{subequations}
\begin{eqnarray}
&&\partial\ldb f,g\rdb\big((a_m,\cdots,a_1),(b_n,\cdots,b_1)\big)\nonumber\\
&=& \sum_i\sum_r(-1)^{|a_m|+\cdots+|a_{i+r}|}
 \ldb f,g\rdb\big((a_m,\cdots, \ov m_r(a_{i+r-1},\cdots,a_{i}),\cdots, a_1),(b_n,\cdots, b_1)\big)\;\;\;\;\label{eq1a}\\
&+&\sum_j\sum_s(-1)^{|u|+|b_n|+\cdots+|b_{s+j}|}
\ldb f,g\rdb\big((a_m,\cdots, a_1),(b_n,\cdots, \ov m_s(b_{j+s-1},\cdots, b_j),\cdots, b_1)\big),\;\;\;\;\;\;\;\;\;\label{eq1b}
\end{eqnarray}
\end{subequations}
while
\begin{subequations}
\begin{eqnarray}
&&
(\ldb\partial f,g\rdb+(-1)^{|f|}\ldb f,\partial g\rdb)((a_m,\cdots,a_1),(b_n,\cdots,b_1))\nonumber\\
&=&\sum_{i,j}\sum_p(-1)^{\varepsilon_{ij}}\mathrm{sgn}(p,p^*)(\partial f)(b_n,\cdots, b_j,p,a_{i-1},\cdots,a_1) g(a_m,\cdots,a_i,p^*,b_{j-1},\cdots,b_1)\nonumber\\
&+ &\sum_{i,j}\sum_p(-1)^{\varepsilon_{ij}+|f|}\mathrm{sgn}(p,p^*) f(b_n,\cdots, b_j,p,a_{i-1},\cdots,a_1) (\partial g)(a_m,\cdots,a_i,p^*,b_{j-1},\cdots,b_1)\nonumber\\
&=&\sum_{i,j}\sum_p(-1)^{\varepsilon_{ij}}\mathrm{sgn}(p,p^*) f(\ov m(b_n,\cdots, b_j,p,a_{i-1},\cdots,a_1))  g(a_m,\cdots,a_i,p^*,b_{j-1},\cdots,b_1)\label{eq2a}\\
&+&\sum_{i,j}\sum_p(-1)^{\varepsilon_{ij}+|f|}\mathrm{sgn}(p,p^*) f(b_n,\cdots, b_j,p,a_{i-1},\cdots,a_1) g(\ov m (a_m,\cdots,a_i,p^*,b_{j-1},\cdots,b_1)).\;\;\;\;\;\;\;\;\label{eq2b}
\end{eqnarray}
\end{subequations}
Comparing the above two equations, one sees that
summand (\ref{eq2a}) contains more terms than (\ref{eq1a}), which are in the form
\begin{equation}\label{ext1}
f(b_n,\cdots, b_k, \ov m_r(b_{k-1},\cdots, b_j,p,a_{i-1},\cdots,a_{\ell}),a_{\ell-1},\cdots, a_1)\cdot g(a_m,\cdots,a_i,p^*,b_{j-1},\cdots,b_1).
\end{equation}
Similarly, summand (\ref{eq2b}) contains more terms than (\ref{eq1b}), in the form
\begin{equation}\label{ext2}
f(b_n,\cdots, b_j,p,a_{i-1},\cdots,a_1)\cdot g(a_m,\cdots,a_{\ell},\ov m_r(a_{\ell-1},\cdots, a_i,p^*,b_{j-1},\cdots, b_k),
b_{k-1},\cdots,b_1).
\end{equation}
We claim these two types of terms (\ref{ext1}) and (\ref{ext2}) cancel with each other. In fact,
(\ref{ext1}) equals
\begin{eqnarray*}
&&\sum_{p}\sum_{q}f(b_n,\cdots, b_k, \epsilon(q,b_{k-1},\cdots, b_j,p,a_{i-1},\cdots,a_{\ell})q^*,a_{\ell-1},\cdots, a_1)\\
&&\quad\quad\quad\quad\quad\quad\quad\quad
\cdot g(a_m,\cdots,a_i,p^*,b_{j-1},\cdots,b_1)\\
&=&\sum_{p}\sum_{q}f(b_n,\cdots, b_k,q^*,a_{\ell-1},\cdots, a_1)\\
&&\quad\quad\quad\quad\quad\quad\quad\quad
\cdot g(a_m,\cdots,a_i,\epsilon(q,b_{k-1},\cdots, b_j,p,a_{i-1},\cdots,a_{\ell})p^*,b_{j-1},\cdots,b_1)\\
&=&\sum_{p}\sum_{q}(-1)^{\nu}f(b_n,\cdots, b_k,q^*,a_{\ell-1},\cdots, a_1)\\
&&\quad\quad\quad\quad\quad\quad\quad\quad
\cdot g(a_m,\cdots,a_i,\epsilon(p,a_{i-1},\cdots,a_{\ell},q, b_{k-1},\cdots, b_j)p^*,b_{j-1},\cdots,b_1)\\
&=&\sum_{p}\sum_{q}(-1)^{\nu}f(b_n,\cdots, b_k,q^*,a_{\ell-1},\cdots, a_1)\\
&&\quad\quad\quad\quad\quad\quad\quad\quad
\cdot g(a_m,\cdots,a_i,\ov m_r(a_{i-1},\cdots,a_{\ell},q, b_{k-1},\cdots, b_j),b_{j-1},\cdots,b_1),
\end{eqnarray*}
which is exactly (\ref{ext2}) after re-indexing the subscripts. In the above expression,
the second equality holds due to the cyclicity assumption \eqref{cyclic_inv_counting},
and also $\nu=(|a_{i-1}|+\cdots+|a_\ell|+|p|)(b_{k-1}+\cdots+|b_j|+|q|)$.
\end{proof}

\begin{remark}
The above calculation looks similar to the double bracket for cyclic algebras
(or more generally Calabi-Yau A$_\infty$
categories) presented in \cite{BCER}; however, Proposition \ref{main_lemma} is slightly more general.
The difference is that here we did not assume the Calabi-Yau cyclicity condition (\ref{CY_cond})
on $\mathcal A $; it is replaced by the cyclicity condition (\ref{cyclic_inv_counting}).
Any Calabi-Yau A$_\infty$ category $\mathcal A$ satisfies conditions
of Proposition \ref{main_lemma}:
if we choose a basis $\{e_i\}$ for $\mathrm{Hom}(A,B)$, then by the non-degenerate pairing
we automatically get a basis $\{e_i^*\}$, the dual basis of $\{e_i\}$, for $\mathrm{Hom}(B, A)$,
and then (\ref{CY_cond}) becomes (\ref{cyclic_inv_counting}).
Moreover, all formulas in Proposition \ref{main_lemma} and in its proof
do not depend on such choice of basis.
On the other hand, these two cyclicity conditions are not equivalent; as we shall see,
Fukaya categories (to be shown below) satisfy (\ref{cyclic_inv_counting}), but not (\ref{CY_cond})
at least in the naive way (see Remark \ref{rmk_cyclicity} below).
\end{remark}

Now suppose $(A, \ldb-,-\rdb)$ is a (possibly complete) DG double Poisson algebra of degree $d$.
Let $\mu:A \hat\otimes A \to A$ denote the multiplication on $A$, and let
$\{-,-\}:=\mu \circ \ldb-,-\rdb: A\hat\otimes A \to A$.
The following is due to Van den Bergh.

\begin{corollary}[Van den Bergh]\label{cor_VdB}
If $A$ is a (possibly complete) DG double Poisson algebra of degree $d$, then
$\{-,-\}$ makes the graded commutator quotient space
$A_\natural=A/[A,A]$ into a DG Lie algebra of degree $d$.
\end{corollary}

\begin{proof}
See Van den Bergh \cite[Corollary 2.4.6]{VdB} for the case of $A$ without differential.
Now suppose $A$ admits a differential $\partial$,
which respects both $\mu$ and $\ldb-,-\rdb$.
Then $\partial$ descends to $A_{\natural}$,
which makes $A_{\natural}$ into a complex (a DG module over $k$).
Moreover, since $\partial$ commutes with $\mu$ and $\ldb-,-\rdb$,
it commutes with $\{-,-\}$ as well,
and we then have that the commutator subspace $[A,A]$ is closed under both $\{-,-\}$ and $\partial$.
It then follows that $(A_{\natural},\{-,-\})$ is well-defined and is a DG Lie algebra of degree $d$.
\end{proof}

\section{Fukaya category of exact symplectic manifolds}\label{Sect_Fuk}

In this section we first recall some necessary ingredients about Fukaya categories,
then prove the main theorem, and after that, relate it to string topology in the case of cotangent bundles.

\subsection{Construction of the Fukaya category}\label{fukcat}

In this subsection we briefly recall the construction of the Fukaya category for
\textit{exact} symplectic manifolds. The complete treatment can be found in Seidel \cite{Seidel}.
The construction of the Fukaya category
on a general symplectic manifold is given in Fukaya \cite[Chapter 1]{Fukaya} and Fukaya et. al. \cite{FOOO}.
Most results in our situation are now well recognized, and hence are cited without proof,
but with precise and concrete references.
We here follow Seidel.

Intuitively, the Fukaya category $\mathrm{Fuk}(M)$ of $M$
is defined as follows: the objects are Lagrangian submanifolds in $M$;
suppose $L_1,L_2$ are two transversal objects,
$\Hom(L_1,L_2)$, called the \textit{Floer cochain complex},
is spanned by the transversal intersection points of $L_1$ and $L_2$, and
for $n$ objects $L_1,\cdots, L_{n+1}$, assume they are pairwisely transversal, then
$$\ov m_n:\ov{\Hom}(L_{n},L_{n+1})\otimes\cdots\otimes\ov{\Hom}
(L_{2},L_3)\otimes\ov{\Hom}(L_1,L_2)\to\ov{\Hom}(L_1,L_{n+1})$$
is given by counting pseudo-holomorphic disks
whose boundary lying in $L_1,L_2,\cdots, L_{n+1}$.
More precisely, if $a_1\in\Hom(L_1,L_2),\cdots,a_n\in\Hom(L_n,L_{n+1})$,
$$\ov m_n(\ov a_n,\cdots,\ov a_2,\ov a_1)=
\sum_{a\in L_1\cap L_{n+1}}\#\mathcal M(a, a_n,\cdots,a_1)\cdot\ov a,$$
where $\#\mathcal M(a, a_n,\cdots,a_1)$ is the counting of
the moduli space of pseudo-holomorphic disks
with $n+1$ (anti-clockwise) cyclically ordered marked points in its boundary, such that
these marked points are mapped onto $a_{n},\cdots,a_1,a$ and that the
rest of the boundary lie in $L_1,L_2,\cdots, L_{n+1}$.

The A$_\infty$ relations (equation (\ref{higher_htpy}))
follow from the compactification of $\mathcal M(a_{n+1},a_n,\cdots,a_1)$,
where those pseudo-holomorphic disks with all possible
``bubbling-off'' disks are added. More precisely,
the compactification of $\mathcal M(a_{n+1},a_n,\cdots,a_1)$ is a stratified space
whose
codimension one strata consists of
\begin{eqnarray}\label{Fukaya-boundary}
\bigcup_{1\le i<j\le n+1}\bigcup_{b\in {L}_i\cap {L}_{j}}
\mathcal{M}(b^*,a_{j-1},\cdots,a_{i})\times\mathcal{M}
(a_{n+1},\cdots, a_{j}, b, a_{i-1},\cdots, a_{1}).
\end{eqnarray}
Now suppose $\mathcal M(-)$ is one dimensional,
then its boundary has even number of components,
and therefore the number of these components is zero if we take
the coefficients of the Floer cochain complex to be
$\mathbb Z_2$.
This exactly corresponds the A$_\infty$ relations for Fukaya category, that is,
(\ref{Fukaya-boundary}) gives (\ref{higher_htpy}) and vise versa.

This is a very rough description of the construction of the Fukaya category.
It is only partially defined in the sense
that we have assumed that all Lagrangian submanifolds are pairwisely transversal;
also, the Floer cochain complexes thus described are only $\mathbb Z_2$ graded and
with only $\mathbb Z_2$ coefficients.
To make the Fukaya category be fully defined
and be graded over $\mathbb Z$ with arbitrary field coefficients, we have to introduce the following concepts.

\subsubsection{Exact symplectic manifolds and admissible Lagrangian submanifolds}\label{SS_assumpt}

A symplectic manifold $(M^{2d},\omega)$ is said to be \textit{exact} if $\omega=d\eta$ for some 1-form $\eta$.
An \textit{exact symplectic manifold with boundary} is a
quadruple $(M,\omega,\eta,J)$, where $M$ is a compact $2d$ dimensional manifold with
boundary, $\omega$ is a symplectic 2-form on $M$, $\eta$ is a
1-form such that $\omega=d\eta$ and $J$ is a $\omega$-compatible
almost complex structure. These data also satisfy the
following two convexity conditions:
\begin{itemize}\item[--] The negative Liouville
vector field defined by
$\omega(\cdot,X_\eta)=\eta$
points strictly inwards along the boundary of $M$;
\item[--] The boundary of $M$ is weakly $J$-convex, which means that any
pseudo-holomorphic curves cannot touch the boundary unless they are
completely contained in it.
\end{itemize}
In the following, for a symplectic manifold $(M^{2d},\omega)$ with or without boundary, we shall always assume $c_1(M)=0$.

A $d$-dimensional submanifold $L\subset M$ is
called \textit{Lagrangian} if $\omega|L=0$. In the following we will assume $L$ is
\textit{closed} and is disjoint from the boundary of $M$. $L$ is called {\it
exact} if $\eta|L$ is an exact 1-form.
In the following, we shall always assume $L$ is \textit{admissible}, namely,
(1) $\eta|_L$ is exact;
(2) $L$ has vanishing Maslov class; and
(3) $L$ is spin.

\begin{example}[Cotangent bundles]\label{ex_cot}
Let $N$ be a simply connected, compact spin manifold. Let $T^*N$ be the cotangent bundle of $N$
with the canonical symplectic structure. The cotangent bundle of $N$ is an exact symplectic manifold.
In particular, $N$, viewed as the zero section of $T^*N$, is an admissible Lagrangian submanifold.
\end{example}

\subsubsection{Construction of the Fukaya category}

From now on, we always assume $M$ is an exact symplectic manifold with $c_1(M)=0$, and
all Lagrangian submanifolds to be considered are compact and admissible. With these assumptions, we have:
\begin{enumerate}
\item[$-$] there exists a (time-dependent) Hamiltonian function on $M$ such that
for each pair of Lagrangian submanifolds $L_0$ and $L_1$,
$\phi(L_0)$ intersects $L_1$ transversally, where $\phi$ is the corresponding
Hamiltonian isotopy of $M$ (see Seidel
\cite[Lemma 9.5]{Seidel} for the existence of such Hamiltonian functions). The orbits of the intersection points
are called {\it Hamiltonian chords}, which span $\mathrm{Hom}(L_0, L_1)$ over $k$.
Moreover, after choosing the Hamiltonian function appropriately,
if $p$ is Hamiltonian chord in $\mathrm{Hom}(L_0, L_1)$, then the same orbit but with
the opposite direction, denoted by $p^*$, lies in $\mathrm{Hom}(L_1, L_0)$;

\item[$-$] there exists a ``grading" on the Lagrangian submanifolds (due to Seidel, see \cite[\S12a]{Seidel}),
which then gives a $\mathbb Z$-grading on the Hamiltonian chords, and
\begin{equation}\label{sum_gradings}
\mathrm{grading}(p)+\mathrm{grading}(p^*)=\dim M/2
\end{equation}
(for a proof of this equation, see Fukaya \cite[Lemma 2.27]{Fukaya});

\item[$-$] the equations for the pseudo-holomorphic disks will be perturbed
to be compatible with the Hamiltonian function (see \cite[\S8f]{Seidel}), and
there is a coherent orientation and compactification on the associated moduli spaces such that
(\ref{Fukaya-boundary}) holds (see \cite[\S9l and \S12b]{Seidel}).
\end{enumerate}

In summary, the conditions in \S\S\ref{SS_assumpt} guarantee that the Floer cochain complex
is defined over a field of characteristic zero and is $\mathbb Z$-graded, and that
the moduli spaces involved are oriented in a coherent way such that the A$_\infty$
hierarchy equations are satisfied.

\begin{theorem}[Fukaya, Seidel]
Suppose $M$ is an exact symplectic manifold with $c_1(M)=0$
and possibly with contact type boundary.
Suppose $ L_1, L_2,\cdots, L_{n+1}$ are admissible graded
Lagrangian submanifolds, and $a_i\in\Hom( L_i, L_{i+1})$ are Hamiltonian chords,
$i=1,2,\cdots, n$. Define
$$\begin{array}{cccl}
 m_n:&{\Hom}
( L_n, L_{n+1})\otimes\cdots\otimes{\Hom}( L_1, L_2)&\longrightarrow&
{\Hom}( L_1, L_{n+1})\\
&(a_n,\cdots, a_2, a_1)&\longmapsto&\displaystyle\sum_{a\in \Hom( L_1, L_{n+1})}\#
{\mathcal M}(a^*,a_n, \cdots, a_1)\cdot a,
\end{array}
$$
for $n=1,2,\cdots$,
where $a$ runs over the set of Hamiltonian chords connecting $L_1$ and $L_{n+1}$.
Then the set of admissible Lagrangian submanifolds
and the Floer cochain complexes among them together with
$\{ m_n\}$ defined above form a cohomologically unital A$_\infty$ category,
called the \textit{Fukaya category} of $M$, and is denoted by
$\mathrm{Fuk}(M)$.
\end{theorem}

\begin{proof}
This is proved by Seidel in \cite[Proposition 12.3]{Seidel}.
\end{proof}

\subsection{Proof of the main theorem}

In the last subsection we have briefly recalled the construction of the Fukaya category
of an exact symplectic manifold. Now we are ready to show:

\begin{theorem}[Theorem \ref{main_thm}]\label{main_thm2}
Suppose $M$ is an exact symplectic $2d$-manifold with $c_1(M)=0$ and possibly with contact type boundary.
Denote by $\mathrm{Fuk}(M)$ the Fukaya category of $M$ and by
$\tilde{B}(\mathrm{Fuk}(M))^{\vee}$
the dual DG algebra of the reduced bar construction of $\mathrm{Fuk}(M)$.
Define $$\ldb-,-\rdb:\tilde{B}(\mathrm{Fuk}(M))^{\vee}
\hat\otimes \tilde{B}(\mathrm{Fuk}(M))^{\vee}
\to
 \tilde{B}(\mathrm{Fuk}(M))^{\vee}
\hat\otimes
 \tilde{B}(\mathrm{Fuk}(M))^{\vee}$$
by the following formula: for homogeneous
$f,g\in \tilde{B}(\mathrm{Fuk}(M))^{\vee}$,
\begin{eqnarray}
&&\ldb f,g\rdb((a_m,\cdots,a_1),(b_n,\cdots,b_1))\nonumber\\
&=&\sum_{i=1}^{m+1}\sum_{j=1}^{n+1}\sum_{p}(-1)^{\varepsilon_{ij}}\mathrm{sgn}(p,p^*)
f(b_n,\cdots, b_j,p,a_{i-1},\cdots,a_1) g(a_m,\cdots,a_i,p^*,b_{j-1},\cdots,b_1), \nonumber
\end{eqnarray}
where
\begin{eqnarray*}
(a_m,\cdots,a_1)&\in&
\ov{\mathrm{Hom}}(L_{m},L_{m+1})\otimes\cdots\otimes\ov{\mathrm{Hom}}(L_1,L_2),\\
(b_n,\cdots,b_1)
&\in& \ov{\mathrm{Hom}}(L_{n}', L_{n+1}')\otimes\cdots\otimes\ov{\mathrm{Hom}}(L_1',L_2'),
\end{eqnarray*}
and $p$ runs over the set of Hamiltonian chords connecting
$L_j'$ and $L_i$.
Then $\ldb-,-\rdb$ defines a DG double Poisson algebra of degree $2-d$ on
$\tilde{B}(\mathrm{Fuk}(M))^{\vee}$.
\end{theorem}

\begin{proof}
From the previous subsections, we observe that there are two key facts about Fukaya categories:
\begin{enumerate}
\item[--]
for each pair of admissible Lagrangian submanifolds $L_0, L_1$, there is
a \textit{canonical} basis for $\mathrm{Hom}( L_0, L_1)$ and $\mathrm{Hom}(L_1, L_0)$,
which are the Hamiltonian chords connecting them.
If $p$ is a Hamiltonian chord connecting $L_0$ and $L_1$,
then the same $p$ with direction reversed, denoted by $p^*$,
connects $L_1$ and $L_0$, and Fukaya proved that their gradings satisfy (\ref{sum_gradings});

\item[--] the moduli space $\mathcal M(a_n,\cdots, a_1, a_0)$ of pseudo-holomorphic disks
is coherently oriented such that
the counting $\#\mathcal M(a_n,\cdots, a_1,a_0)$ is cyclically invariant, that is,
$\#\mathcal M(a_n,\cdots, a_1,a_0)=(-1)^{|a_0|(|a_1|+\cdots+|a_n|)} \#\mathcal M(a_0,a_n,\cdots, a_{1})$.
\end{enumerate}
This means that $\mathrm{Fuk}(M)$
satisfies the conditions of Proposition \ref{main_lemma}, from which
the theorem follows.
\end{proof}

\begin{remark}\label{rmk_cyclicity}
One may formally define a pairing
$$\langle-,-\rangle:\Hom( L_0, L_1)
\otimes
\Hom( L_0, L_1)\to k$$
by
$$\langle p,q\rangle=\left\{\begin{array}{cl}(-1)^{|p|-1}\mathrm{sgn}(p,q),&\mbox{if\;}q=p^*,\\
0,&\mbox{otherwise,}\end{array}\right.$$
and extend it linearly to all morphism space (note that here we have not
shifted the degree down yet). It is graded symmetric; however,
it does not satisfy the Calabi-Yau condition
(\ref{CY_cond}).
\end{remark}

\begin{corollary}\label{main_cor}
Suppose $M$ is an exact symplectic $2d$-manifold with $c_1(M)=0$ and
possibly with contact type boundary.
Denote by $\mathrm{Fuk}(M)$ the Fukaya category of $M$.
Then $\mathrm{HC}^\bullet(\mathrm{Fuk}(M) )$ has a degree $2-d$ graded Lie algebra structure.
\end{corollary}

\begin{proof}
This is a combination of Proposition \ref{Prop_Quillen}, Theorem \ref{main_thm2} and Corollary \ref{cor_VdB}.
More precisely, Theorem \ref{main_thm2} says that $\tilde{B}(\mathrm{Fuk}(M))^{\vee}$ is equipped with a DG double Poisson
bracket of degree $2-d$, and then by Corollary \ref{cor_VdB}, its commutator quotient space is a DG Lie algebra of
degree $2-d$, where the latter, by Proposition \ref{Prop_Quillen}, is exactly the
cyclic cochain complex $\mathrm{CC}^\bullet(\mathrm{Fuk}(M))$.
By taking homology, we see that $\mathrm{HC}^\bullet(\mathrm{Fuk}(M))$ thus have
a degree $2-d$ graded Lie algebra structure.
\end{proof}

Let us say some more words about this Lie algebra.
In \cite{CB} Crawley-Boevey introduced a notion of
{\it $\mathrm H_0$-Poisson structure}, which is defined as follows:
Suppose $A$ is an associative algebra; an $\mathrm H_0$-Poisson structure on $A$ is a Lie  bracket $\{-,-\}$ on
$A/[A, A]$ such that
$$\overline a\mapsto \{\overline a, -\},\quad \overline a\in A/[A,A]$$
is induced by a derivation $d_{a}: A\to A$. The significance of this notion is the following.

\begin{theorem}[Crawley-Boevey \cite{CB}]
Let $A$ be an associative algebra over an algebraically closed field $k$ of characteristic zero.
If $A$ admits an $\mathrm H_0$-Poisson structure, then there is a unique Poisson structure
on the coordinate ring of the
isomorphism classes of $n$-dimensional $k$-representations
$\mathrm{Rep}_{n}(A)/\!/\mathrm{GL}(n)$, for all $n\in\mathbb N$,
such that
the trace map
$$
\begin{array}{cccl}
\mathrm{Tr:} & A/[A,A]&\longrightarrow& \mathrm{Rep}_{n}(A)/\!/\mathrm{GL}(n)\\
&\ov a&\longmapsto&\big\{ \rho\mapsto \mbox{trace}(\rho(a))\big\}
\end{array}
$$
is a map of Lie algebras.
\end{theorem}

\begin{proof}
See Crawley-Boevey \cite[Theorem 2.5]{CB}.
\end{proof}

For an associative algebra $A$, if
it admits a double Poisson bracket $\ldb-,-\rdb$, then from Van
den Bergh's result (Corollary \ref{cor_VdB}) one immediately obtains an $\mathrm H_0$-Poisson structure on $A$.
In particular, Van den Bergh (\cite{VdB}) showed that
the path algebra of a doubled quiver has a double Poisson structure,
which then induces a Poisson structure on the representation scheme of this doubled quiver,
and hence recovers an important result of Ginzburg \cite{Gin} and
Bocklandt-Le Bruyn \cite{BL}.

The work of Crawley-Boevey and Van den Bergh cited above was later
further studied in \cite{BCER} based on the work \cite{BKR},
where DG algebras and DG representations are studied.
The current work may be viewed as a continuation of \cite{BCER},
with an aim to the understanding of some algebraic structures in symplectic topology.
At present, we are not able to describe the representation theory of a general Fukaya category.
Nevertheless, from the work of \cite{BS} and \cite{Seidel2}, one sees that for some special class of symplectic manifolds,
the sub-category of vanishing cycles is very much related to the representation theory of the associated quivers.
We hope to turn to this point in the near future.

\subsection{Example of cotangent bundles}

In this subsection, we compare the previous results with string topology in the case
of cotangent bundles.
Suppose $N$ is a smooth $d$-manifold, and denote by $LN$ the free
loop space of $N$. In \cite{CS1} Chas and Sullivan showed
that the $S^1$-equivariant homology $\mathrm{H}_\bullet^{S^1}(LN)$
of $LN$ has the structure of a degree $2-d$ Lie algebra structure; later in \cite{CS2} they further show that,
by modulo the constant loops, $\mathrm{H}_\bullet^{S^1}(LN, N)$ is in fact an involutive Lie bialgebra.
Ever since its first appearance,
a lot of efforts have been made by mathematicians
in trying to understand such Lie (bi)algebra structure.
In the following we briefly
show that the Lie algebra that we obtained in the previous subsection
is very much similar to that of Chas and Sullivan.

First, we recall the definition
of A$_\infty$ functors and their properties; a much complete treatment of this topic can also be found in
\cite{Hasegawa,Seidel}.

\begin{definition}[A$_\infty$ functor; {\it cf.} \cite{Seidel} \S1b]
Suppose
$\mathcal A $ and $\mathcal B$ are two A$_\infty$ categories.
An A$_\infty$ functor $F:\mathcal A \to\mathcal B$ consists
of a map $F:\mathrm{Ob}(\mathcal A )\to\mathrm{Ob}(\mathcal B)$
and a sequence of multilinear maps
$$F^n: \ov{\mathrm{Hom}}_{\mathcal A }(A_{n}, A_{n+1})\otimes\cdots\otimes
\ov{\mathrm{Hom}}_{\mathcal A }(A_1, A_2)\to
\ov{\mathrm{Hom}}_{\mathcal B}(F(A_1), F(A_{n+1}))$$
of degree $0$ such that
\begin{multline}\label{Ainfty_functor}
\sum_{r}\sum_{s_1,s_2,\cdots, s_r}\ov m^{\mathcal B}_r
(F^{s_r}(a_n,\cdots, a_{n-s_r+1}),\cdots, F^{s_1}(a_{s_1},\cdots, a_1)) \\
=\sum_{p,q}(-1)^{|a_1|+\cdots+|a_{p}|}
 F^{n-q+1}(a_n, \cdots, a_{p+q+1},\ov m^{\mathcal A }_{q}(a_{p+q},\cdots, a_{p+1}), a_p,\cdots, a_1),
\end{multline}
where on the left hand side the sum is over all $r\ge 1$ and all partitions $s_1+\cdots+s_r=n$.
\end{definition}

In other words, equation (\ref{Ainfty_functor}) can be read as
a DG coalgebra map on the (reduced) bar constructions:
\begin{equation}\label{coalgebra_map}
B(F): \tilde B(\mathcal A )\to \tilde B(\mathcal B).
\end{equation}
If $F:\mathcal A\to\mathcal B$ is an A$_\infty$ functor, then one can associate
a functor $\mathrm H(F): \mathrm H(\mathcal A)\to\mathrm H(\mathcal B)$, called the {\it cohomology level functor},
sending objects of $\mathrm H(\mathcal A)$
to objects of $\mathrm H(\mathcal B)$ as that of $F$,
and sending $[a]\in \mathrm H_\bullet(\mathrm{Hom}(A,B), m_1^{\mathcal A})$
to $[F^1(a)]\in\mathrm H_\bullet(\mathrm{Hom}(F(A),F(B)), m_1^{\mathcal B})$.

Now suppose $\mathcal A$ and $\mathcal B$ are cohomologically unital.
An A$_\infty$ functor $F:\mathcal A\to\mathcal B$ is called a {\it quasi-equivalence} if
the associated cohomology level functor $\mathrm H(F)$ is an equivalence; it is a {\it quasi-isomorphism} if
$\mathrm H(F)$ is an isomorphism.
For the Fukaya category
of cotangent bundles (Example \ref{ex_cot}),
the following result is obtained by Fukaya-Seidel-Smith and Nadler independently:

\begin{theorem}[Fukaya-Seidel-Smith and Nadler]\label{FSSN}
Let $N$ be a simply-connected, compact spin manifold, and let $T^*N$ be its cotangent bundle.
Then there is a quasi-equivalence
of A$_\infty$ categories
$$\Phi: \mathrm{Fuk}(T^*N)
\to\mathrm{Hom}(N, N),
$$
where the latter is the sub Fukaya category of $T^*N$ with one object $N$
(that is, the Floer cochain complex of $N$).
\end{theorem}

\begin{proof}
See Fukaya-Seidel-Smith \cite[Theorem 1]{FSS}
and Nadler \cite[Theorem 1.3.1]{Nadler}.
A further discussion can be found in Fukaya-Seidel-Smith (\cite{FSS2}).
\end{proof}

On the other hand, we have the following theorem, which is usually called the PSS (Piunikhin-Salamon-Schwarz) isomorphism
in literature.

\begin{theorem}[PSS isomorphism]\label{thm_PSS}
Let $N$ be as in previous theorem. Then
the Floer cochain complex $\mathrm{Hom} (N, N)$
is quasi-isomorphic to the singular cochain complex $C^\bullet (N)$ as A$_\infty$ algebras.
\end{theorem}

\begin{proof}
There have been several proofs for this result; see, for example, Abouzaid \cite[Theorem 1.1]{Ab09}.
In fact, in \cite{Ab09} the theorem is proved for any compact smooth manifold $N$.
\end{proof}

As a corollary to the above two theorems, we have the following result.

\begin{corollary}
Let $N$ be a simply-connected, compact spin $d$-manifold.
Then the cyclic cohomology of $\mathrm{Fuk}(T^*N)$ is isomorphic
to the cyclic cohomology of $C^\bullet(N)$, which induces on the latter a degree $2-d$ Lie algebra structure.
\end{corollary}

\begin{proof}
This corollary follows from Corollary \ref{main_cor}, Theorems \ref{FSSN}, \ref{thm_PSS}
and the following lemma saying that cyclic (co)homology is invariant under quasi-equivalences/quasi-isomorphisms.
\end{proof}

\begin{lemma}\label{inv_cyclic}
Cyclic cohomology is invariant under quasi-equivalences, that is, for two quasi-equivalent
A$_\infty$ categories $\mathcal A $ and $\mathcal B$, we have
$$
\mathrm{HC}^\bullet(\mathcal A )\cong\mathrm{HC}^\bullet(\mathcal B).
$$
\end{lemma}

To prove this lemma, let us first recall the Yoneda embedding for A$_\infty$ categories.
We start with A$_\infty$ modules ({\it cf.} \cite[\S1j]{Seidel}).

\begin{definition}[A$_\infty$ module]
Suppose $\mathcal C$ is an A$_\infty$ category.
A $\mathcal C$-module is an A$_\infty$ functor from $\mathcal C$ to the DG category of complexes over $k$,
where the latter is viewed as an A$_\infty$ category.
\end{definition}

Since the category of complexes over $k$ is a DG category, all $\mathcal C$-modules form
a DG category as well, which is called the module category of $\mathcal C$, and
is denoted by $\mathrm{mod}(\mathcal C)$.
Moreover,
there is a natural functor (called the {\it Yoneda embedding})
$$
\mathcal I:  \mathcal C \longrightarrow \mathrm{mod}(\mathcal C)
$$
which maps any object, say $Y$ in $\mathcal C$, to
the object $\mathrm{Hom}(-, Y)$ in $ \mathrm{mod}(\mathcal C )$,
and maps morphisms in $\mathcal C$ to morphisms in $\mathrm{mod}(\mathcal C)$ in
the natural way
({\it cf.} \cite[\S1l]{Seidel}).
Denote by $\mathcal{IC}$ the images of $\mathcal C$ under the Yoneda embedding,
then by some technical discussion on the unit, Seidel proves the following

\begin{lemma}[Seidel \cite{Seidel} Corollary 2.14]\label{lemma_qi1}
Any cohomologically unital A$_\infty$ category $\mathcal C$ is canonically quasi-isomorphic to the
strictly unital
DG category $\mathcal {IC}$ via the Yoneda embedding.
\end{lemma}

Now suppose $F:\mathcal C\to\mathcal C'$ is an A$_\infty$ functor, it induces
a DG functor (the {\it pull-back functor})
$$F^*:\mathrm{mod}(\mathcal C') \to\mathrm{mod}(\mathcal C)$$
which maps
a $\mathcal C'$-module, say $M$, to a $\mathcal C$-module $F^*(M)$ given as follows: for any object
$X\in\mathrm{Ob}(\mathcal C)$, $F^*(M)(X)=M(F(X))$ (for more details see \cite[\S1k]{Seidel}).
Moreover, the Yoneda embedding is natural in the following sense:

\begin{lemma}[Seidel \cite{Seidel} Diagram 2.13]\label{lemma_qi2}
Suppose $F:\mathcal C\to\mathcal C'$ is a cohomologically full and faithful A$_\infty$ functor, then
the following diagram is commutative
$$
\xymatrixcolsep{4pc}
\xymatrix{
\mathcal C\ar[r]^-{F}\ar[d]^{\mathcal I}&\mathcal C'\ar[d]_{\mathcal I}\\
\mathrm{mod}(\mathcal C)&\ar[l]_{F^*}\mathrm{mod}(\mathcal C').
}$$
In particular, if $F$ is a quasi-equivalence, then $F^*: \mathcal{IC'}\to\mathcal{IC}$ is an quasi-equivalence
of DG categories.
\end{lemma}

\begin{proof}[Proof of Lemma \ref{inv_cyclic}]
The proof consists of two steps. The first step is to show that for any A$_\infty$ category $\mathcal C$,
the Yoneda embedding $\mathcal I:\mathcal C\to\mathcal {IC}$ of Lemma \ref{lemma_qi1} induces an isomorphism on
their cyclic (co)homology.
In fact, since $\mathcal I$ is an A$_\infty$ functor,
we have a map of DG coalgebras (see (\ref{Ainfty_functor}))
$$B(\mathcal I): \tilde B(\mathcal C){\to} \tilde B(\mathcal{IC}).$$
By a well known result which says any quasi-isomorphism of
A$_\infty$ categories has an inverse up to homotopy (for a proof
see \cite[Corollary 1.14]{Seidel}),
$B(\mathcal I)$ in fact has an inverse induced from that of $\mathcal I$,
and their compositions are homotopic to identity on each side.
This means we in fact have a homotopy equivalence of DG coalgebras
$$B(\mathcal I): \tilde B(\mathcal C)\stackrel{\simeq}{\to} \tilde B(\mathcal{IC}).$$
Dually, we obtain a homotopy equivalence of DG algebras
$$\tilde B(\mathcal{IC})^{\vee}\stackrel{\simeq}{\to}\tilde B(\mathcal C)^{\vee}.$$
Note that homotopy equivalent
DG algebras induce quasi-isomorphic chain complexes
on their commutator quotient spaces (for a proof
see \cite[Lemma 3.1]{BKR}), and since $\mathcal{IC}$ is
a strictly unital DG category (Lemma \ref{lemma_qi1}),
the commutator quotient space
is the cyclic cohomology of $\mathcal{IC}$ in the usual sense.
Thus we get an isomorphism
\begin{equation}\label{qi_1}
\mathrm{HC}^\bullet(\mathcal {IC})\stackrel{\cong}{\to}\mathrm{HC}^\bullet(\mathcal{C}).
\end{equation}
The second step is, assume $F: \mathcal A\to\mathcal B$ is a quasi-equivalence, then by Lemma \ref{lemma_qi2},
$F^*: \mathcal {IB}\to\mathcal{IA}$ is a quasi-equivalence of DG categories. A theorem of Keller
(see \cite[Theorem 1.5]{Keller}) says that for quasi-equivalent DG categories,
their periodic cyclic homology groups
(respectively negative cyclic, cyclic homology as well as cyclic cohomology groups)
are isomorphic. That is, we have
\begin{equation}\label{qi_2}
\mathrm{HC}(F^*):\mathrm{HC}^\bullet(\mathcal{IB})\cong\mathrm{HC}^\bullet(\mathcal{IA}).
\end{equation}
Combining (\ref{qi_1}) and (\ref{qi_2}) we obtain isomorphisms
\begin{equation*}\mathrm{HC}^\bullet(\mathcal A)\stackrel{(\ref{qi_1})}\cong\mathrm{HC}^\bullet(\mathcal{IA})
\stackrel{(\ref{qi_2})}\cong
\mathrm{HC}^\bullet(\mathcal {IB})\stackrel{(\ref{qi_1})}\cong
\mathrm{HC}^\bullet(\mathcal B).\qedhere\end{equation*}
\end{proof}

It is nowadays also well-known that the cyclic cohomology of $C^\bullet (N)$
is nothing but the $S^1$-equivariant homology of $LN$, where the $S^1$-action
is the rotation of loops:

\begin{theorem}[Jones]\label{thm_Jones}
Suppose $N$ is a simply-connected manifold. Let $LN$ be the free loop space of $N$. Then
we have the following isomorphism
$$\mathrm{HC}^\bullet (C^\bullet(N))\cong \mathrm{H}_\bullet^{S^1}(LN).$$
\end{theorem}

\begin{proof}See
Jones \cite[Theorem A]{Jones}.
\end{proof}

This means that, in the cotangent bundle case, if the base manifold is simply-connected, compact and spin, then
by Theorems \ref{FSSN}, \ref{thm_PSS} and \ref{thm_Jones} together with Lemma \ref{inv_cyclic},
the cyclic cohomology of the Fukaya category of closed Lagrangian submanifolds,
of the Floer cochain complex of the zero section, and of
the singular cochain complex of the zero section, are all isomorphic to the $S^1$-equivariant homology of the
free loop space of the base manifold.
By transporting the Lie algebra obtained in previous subsection to the equivariant homology,
we may summarize the above discussion
into the following theorem:

\begin{theorem}\label{Lie_freeloop}
Let $N$ be a simply-connected, compact spin $d$-manifold. Denote by $LN$ and $T^*N$ the free
loop space and the cotangent bundle of $N$ respectively.
Then the noncommutative Poisson structure on $\mathrm{Fuk}(T^*N)$ given by Theorem \ref{main_thm} induces
a degree $2-d$ Lie algebra on $\mathrm{H}_{\bullet}^{S^1}(LN)$.
\end{theorem}

Now let us go back to string topology. The construction of Chas-Sullivan is sketched as follows:
for two homology classes, say $[\alpha]$ and $[\beta]\in\mathrm{H}_\bullet^{S^1}(LN)$,
suppose they are represented by $\alpha$ and $\beta$,
then we may view them as two families of loops forgetting the marked points (note
that the loops have natural marked points given by their starting points).
Now assume $\alpha$ and $\beta$ are transversal to each other;
to get a Lie bracket on $\alpha$ and $\beta$, we first equip them with the marked points in all possible ways,
which are parametrized by two $S^1$. Then for any $t_1, t_2\in S^1$, consider the intersection of
the marked loci of $\alpha$ at $t_1$ and $\beta$ at $t_2$
and form a new family of loops over the common loci which are the concatenation of loops from
$\alpha$ and $\beta$. As $t_1$ and $t_2$ vary in $S^1$, we in fact get a $(2-d)$ dimensional
family of loops in $LN$. By forgetting the marked points of this new chain,
Chas and Sullivan proved in \cite{CS1,CS2} that it represents a homology class
in $\mathrm H_\bullet^{S^1}(LN)$, which this the so-called Chas-Sullivan string Lie bracket of
$[\alpha]$ and $[\beta]$ (See \cite{CS1,CS2} for more details).

Chas and Sullivan's construction is partially inspired by the work of Goldman \cite{Goldman}.
In this work, Goldman proved that the space spanned by the free homotopy classes of loops,
modulo the constant ones,
forms a Lie algebra, where the Lie bracket is exactly the same as described above. It turns out that the Goldman
Lie bracket is very much similar to the Kontsevich bracket (see \cite{BL,Gin,VdB}). Since the
Lie bracket on the cyclic cohomology of the Fukaya category and hence
on the $S^1$-equivariant homology of $LN$
is directly inspired by Kontsevich, in this sense we may say that the Lie algebra given in Theorem \ref{Lie_freeloop}
is also similar to the one of Chas and Sullivan.

\begin{appendix}

\section{Proof of Lemma \ref{cyclic_sum}}

In this appendix, we prove Lemma \ref{cyclic_sum}, that is, to show
the commutativity of Diagram \eqref{commdiag_cmpx}.
Since we work over a field $k$ of characteristic zero, the horizontal sequences
are exact. The remaining
proof consists of the following two propositions.

\begin{proposition}$b'N=Nb$.
\end{proposition}

\begin{proof}
We consider the action of both sides on the element $(\ov a_{n+1},\cdots \ov a_1)$.
There are two types of summands in $b'N$.
The first one are those terms whose indices appearing in $\ov m_i(\cdots)$ are in decreasing order,
and the second one are the rest; we show they are equal to $Nb'$ and $Nb''$ respectively.

In fact,
for fixed $i\geq 1$ and $k\geq 1$, $N(\ov a_{n+1},\cdots,\ov a_1)=N_n(\ov a_{n+1},\cdots,\ov a_{1})$, which is equal to
\begin{eqnarray*}
& & (\ov a_{n+1},\ov a_n,\ov a_{n-1},\cdots,\ov a_{k+i},\ov a_{k+i-1},\cdots, \ov a_k, \ov a_{k-1},\cdots, \ov a_1)\\
&+& (-1)^{|\ov a_1|(\sum_{l=1}^{n+1}|\ov a_l|-|\ov a_1|)}
(\ov a_1,\ov a_{n+1}, \ov a_n,\cdots, \ov a_{k+i},\cdots, \ov a_k,\ov a_{k-1},\cdots, \ov a_2)\\
&+& \cdots\\
&+& (-1)^{|\ov a_1|(\sum_{l=1}^{n+1}|\ov a_l|-|\ov a_1|)+\cdots+|\ov a_{k-2}|
(\sum_{l=1}^{n+1}|\ov a_l|-|\ov a_{k-2}|)}
(\ov a_{k-2},\cdots,\ov a_{1}, \ov a_{n+1},\ov a_{n},\cdots, \ov a_{k+i}, \cdots, \ov a_k,\ov a_{k-1})\\
&+& (-1)^{|\ov a_1|(\sum_{l=1}^{n+1}|\ov a_l|-|\ov a_1|)+\cdots+|\ov a_{k-1}|(\sum_{l=1}^{n+1}|\ov a_l|-|\ov a_{k-1}|)}
(\ov a_{k-1},\cdots,\ov a_{1}, \ov a_{n+1},\ov a_{n},\cdots, \ov a_{k+i},\ov a_{k+i-1},\cdots, \ov a_k)\\
&+& (-1)^{|\ov a_1|(\sum_{l=1}^{n+1}|\ov a_l|-|\ov a_1|)+\cdots+|\ov a_{k}|(\sum_{l=1}^{n+1}|\ov a_l|-|\ov a_{k}|)}
(\ov a_{k},\cdots,\ov a_{1}, \ov a_{n+1},\ov a_{n},\cdots, \ov a_{k+i},\ov a_{k+i-1},\cdots, \ov a_{k+1})\\
&+&\cdots\\
&+& (-1)^{|\ov a_1|(\sum_{l=1}^{n+1}|\ov a_l|-|\ov a_1|)+\cdots+|\ov a_{k+i-2}|(\sum_{l=1}^{n+1}|\ov a_l|-|\ov a_{k+i-2}|)}
(\ov a_{k+i-2},\cdots,\ov a_{k},\cdots, \ov a_1,\ov a_{n+1},\cdots, \ov a_{k+i-1})\\
&+& (-1)^{|\ov a_1|(\sum_{l=1}^{n+1}|\ov a_l|-|\ov a_1|)+\cdots+|\ov a_{k+i-1}|(\sum_{l=1}^{n+1}|\ov a_l|-|\ov a_{k+i-1}|)}
(\ov a_{k+i-1},\cdots, \ov a_{k},\ov a_{k-1},\cdots, \ov a_1,\ov a_{n+1},\cdots, \ov a_{k+i})\\
&+& (-1)^{|\ov a_1|(\sum_{l=1}^{n+1}|\ov a_l|-|\ov a_1|)+\cdots+|\ov a_{k+i}|(\sum_{l=1}^{n+1}|\ov a_l|-|\ov a_{k+i}|)}
(\ov a_{k+i}, \cdots, \ov a_k,\ov a_{k-1},\cdots, \ov a_1,\ov a_{n+1},\cdots, \ov a_{k+i+1})\\
&+& \cdots\\
&+& (-1)^{|\ov a_1|(\sum_{l=1}^{n+1}|\ov a_l|-|\ov a_1|)+\cdots+|\ov a_{n}|(\sum_{l=1}^{n+1}|\ov a_l|-|\ov a_{n}|)}
(\ov a_n,\cdots \ov a_{k+i},\ov a_{k+i-1},\cdots, \ov a_k,\ov a_{k-1},\cdots, \ov a_1,\ov a_{n+1}).
\end{eqnarray*}
Then the first type of summands in
$b'N(\ov a_{n+1},\cdots,\ov a_1)$ are
\begin{subequations}\label{1st_left}
\begin{eqnarray}
&& (-1)^{\varepsilon_k} (\ov a_{n+1},\ov a_n,\ov a_{n-1},\cdots,\ov a_{k+i},\ov m_i(\ov a_{k+i-1},\cdots ,\ov a_k), \ov a_{k-1},\cdots, \ov a_1)\\
&+& (-1)^{|\ov a_1|(\sum_{l=1}^{n+1}|\ov a_l|-|\ov a_1|)}(-1)^{\varepsilon_k-|\ov a_1|}
(\ov a_1,\ov a_{n+1}, \cdots, \ov a_{k+i},\ov m_i(\ov a_{k+i-1},\cdots, \ov a_k),\ov a_{k-1},\cdots, \ov a_2) \\
&+& \cdots\nonumber\\
&+& (-1)^{|\ov a_1|(\sum_{l=1}^{n+1}|\ov a_l|-|\ov a_1|)+\cdots+|\ov a_{k-2}|(\sum_{l=1}^{n+1}|\ov a_l|-|\ov a_{k-2}|)}
(-1)^{\varepsilon_k-|\ov a_1|-\cdots-|\ov a_{k-2}|} \nonumber\\
&&\quad\quad (\ov a_{k-2},\cdots,\ov a_{1}, \ov a_{n+1},
\ov a_{n},\cdots, \ov a_{k+i},\ov m_i(\ov a_{k+i-1},\cdots, \ov a_k),\ov a_{k-1}) \\
&+& (-1)^{|\ov a_1|(\sum_{l=1}^{n+1}|\ov a_l|-|\ov a_1|)+\cdots+|\ov a_{k-1}|(\sum_{l=1}^{n+1}|\ov a_l|-|\ov a_{k-1}|)}
(-1)^{\varepsilon_k-|\ov a_1|-\cdots-|\ov a_{k-1}|}\nonumber \\
&&\quad\quad (\ov a_{k-1},\cdots,\ov a_{1}, \ov a_{n+1},\ov a_{n},\cdots, \ov a_{k+i},\ov m_i(\ov a_{k+i-1},\cdots, \ov a_k))\\
&+& (-1)^{|\ov a_1|(\sum_{l=1}^{n+1}|\ov a_l|-|\ov a_1|)+\cdots+|\ov a_{k+i-1}|(\sum_{l=1}^{n+1}|\ov a_l|-|\ov a_{k+i-1}|)}
(-1)^{\varepsilon_k+|\ov a_{n+1}|+\cdots+|\ov a_{k+i}|} \nonumber\\
&&\quad\quad (\ov m_i(\ov a_{k+i-1},\cdots, \ov a_{k}),\ov a_{k-1},\cdots, \ov a_1,\ov a_{n+1},\cdots, \ov a_{k+i})
\label{1st_left_e}\\
&+& (-1)^{|\ov a_1|(\sum_{l=1}^{n+1}|\ov a_l|-|\ov a_1|)+\cdots+|\ov a_{k+i}|(\sum_{l=1}^{n+1}|\ov a_l|-|\ov a_{k+i}|)}
(-1)^{\varepsilon_k+|\ov a_{n+1}|+\cdots+|\ov a_{k+i+1}|}\nonumber \\
&&\quad\quad (\ov a_{k+i},\ov m_i(\ov a_{k+i-1},\cdots, \ov a_k),\ov a_{k-1},\cdots, \ov a_1,\ov a_{n+1},\cdots, \ov a_{k+i+1}) \\
&+& \cdots\nonumber \\
&+& (-1)^{|\ov a_1|(\sum_{l=1}^{n+1}|\ov a_l|-|\ov a_1|)+\cdots+|\ov a_{n}|(\sum_{l=1}^{n+1}|\ov a_l|-|\ov a_{n}|)}
(-1)^{\varepsilon_k+|\ov a_{n+1}|} \nonumber\\
&&\quad\quad (\ov a_n,\cdots, \ov a_{k+i}, \ov m_i(\ov a_{k+i-1},\cdots, \ov a_k),\ov a_{k-1},\cdots, \ov a_1,\ov a_{n+1}),
\end{eqnarray}
\end{subequations}
where $\varepsilon_k=|\ov a_{k-1}|+\cdots+|\ov a_1|$.
In $Nb=N(b'+b'')$, there is no contribution for such terms from $b''$,
and $b'(\ov a_{n+1},\cdots, \ov a_1)$ is equal to
$(-1)^{\varepsilon_k}(\ov a_{n+1},\cdots,\ov a_{k+i}, \ov m_i(\ov a_{k+i-1},\cdots,\ov a_k),\ov a_{k-1},\cdots, \ov a_1)$.
The action of $N$ on the latter is in fact the $N_{(n+1)-i}$-action, which is equal to
\begin{subequations}\label{1st_right}
\begin{eqnarray}
&& (-1)^{\varepsilon_k} (\ov a_{n+1},\ov a_n,\ov a_{n-1},\cdots,\ov a_{k+i},\ov m_i(\ov a_{k+i-1},\cdots, \ov a_k), \ov a_{k-1},\cdots, \ov a_1)\\
&+& (-1)^{|\ov a_1|(\sum_{l=1}^{n+1}|\ov a_l|-|\ov a_1|)+|\ov a_1|}
(-1)^{\varepsilon_k}(\ov a_1,\ov a_{n+1}, \cdots, \ov a_{k+i},\ov m_i(\ov a_{k+i-1},\cdots, \ov a_k),\ov a_{k-1},\cdots, \ov a_2)\\
&+& \cdots\nonumber\\
&+& (-1)^{|\ov a_1|(\sum_{l=1}^{n+1}|\ov a_l|-|\ov a_1|)+\cdots+
|\ov a_{k-2}|(\sum_{l=1}^{n+1}|\ov a_l|-|\ov a_{k-2}|)+|\ov a_1|+\cdots+|\ov a_{k-2}|}(-1)^{\varepsilon_k}\nonumber\\
&&\quad\quad
(\ov a_{k-2},\cdots,\ov a_{1}, \ov a_{n+1},\ov a_{n},\cdots, \ov a_{k+i},
\ov m_i(\ov a_{k+i-1},\cdots, \ov a_k),\ov a_{k-1}) \\
&+& (-1)^{|\ov a_1|(\sum_{l=1}^{n+1}|\ov a_l|-|\ov a_1|)+\cdots+
|\ov a_{k-1}|(\sum_{l=1}^{n+1}|\ov a_l|-|\ov a_{k-1}|)+|\ov a_1|+\cdots+|\ov a_{k-1}|}(-1)^{\varepsilon_k}\nonumber\\
&&\quad\quad
(\ov a_{k-1},\cdots,\ov a_{1}, \ov a_{n+1},\ov a_{n},\cdots, \ov a_{k+i},\ov m_i(\ov a_{k+i-1},\cdots, \ov a_k))\\
&+& (-1)^{|\ov a_1|(\sum_{l=1}^{n+1}|\ov a_l|-|\ov a_1|)+
\cdots+|\ov a_{k-1}|(\sum_{l=1}^{n+1}|\ov a_l|-|\ov a_{k-1}|)+|\ov a_1|+\cdots+|\ov a_{k-1}|}(-1)^{\varepsilon_k}\nonumber\\
&& \quad\quad
(-1)^{(|\ov a_{k+i-1}|+\cdots+|\ov a_k|+1)(|\ov a_{k-1}|+\cdots+|\ov a_1|+|\ov a_{n+1}|+\cdots+|\ov a_{k+i}|)}\nonumber\\
&&\quad\quad(\ov m_i(\ov a_{k+i-1},\cdots, \ov a_{k}),\ov a_{k-1},\cdots, \ov a_1,\ov a_{n+1},\cdots, \ov a_{k+i})
\label{1st_right_e}\\
&+& (-1)^{|\ov a_1|(\sum_{l=1}^{n+1}|\ov a_l|-|\ov a_1|)+\cdots+
|\ov a_{k+i}|(\sum_{l=1}^{n+1}|\ov a_l|-|\ov a_{k+i}|)+|\ov a_{n+1}|+\cdots+|\ov a_{k+i+1}|}(-1)^{\varepsilon_k}\nonumber\\
&&\quad\quad
(\ov a_{k+i},\ov m_i(\ov a_{k+i-1},\cdots, \ov a_k),\ov a_{k-1},\cdots, \ov a_1,\ov a_{n+1},\cdots, \ov a_{k+i+1}) \\
&+& \cdots\nonumber\\
&+& (-1)^{|\ov a_1|(\sum_{l=1}^{n+1}|\ov a_l|-|\ov a_1|)+\cdots+
|\ov a_{n}|(\sum_{l=1}^{n+1}|\ov a_l|-|\ov a_{n}|)+|\ov a_{n+1}|}(-1)^{\varepsilon_k}\nonumber\\
&&\quad\quad
(\ov a_n,\cdots, \ov a_{k+i}, \ov m_i(\ov a_{k+i-1},\cdots, \ov a_k),\ov a_{k-1},\cdots, \ov a_1,\ov a_{n+1}).
\end{eqnarray}
\end{subequations}
From the above two expressions
we see that \eqref{1st_left}=\eqref{1st_right} term by term, where, in particular, \eqref{1st_left_e}=\eqref{1st_right_e}
due to the fact that $$\Big(\sum_{l=1}^i |\ov a_{k+l-1}|\Big)^2-\sum_{l=1}^i |\ov a_{k+l-1}|^2$$ is even.
In other words, we have $Nb'=\mbox{first types of summands in}\,b'N$.

Now in $b'N(\ov a_{n+1},\cdots, \ov a_1)$, the rest summands (second type of summands) contain
components $$\ov m_{i+j+1}(\ov a_i,\cdots,\ov a_1,\ov a_{n+1},\cdots,\ov a_{n-j-1})$$ with indices
not in decreasing order, for some fixed $i\geq 1$ and $j\geq 0$. More precisely, the second type of summands in
$b'N(\ov a_{n+1},\ov a_n,\cdots,\ov a_{n-j+2},\ov a_{n-j+1},\ov a_{n-j},\cdots,\ov a_{i+1},\ov a_{i},\cdots,\ov a_1)$
are
\begin{eqnarray}\label{2nd_left}
&&(-1)^{|\ov a_1|(\sum_{l=1}^{n+1}|\ov a_l|-|\ov a_1|)+\cdots+
|\ov a_{i}|(\sum_{l=1}^{n+1}|\ov a_l|-|\ov a_{i}|)}(-1)^{|\ov a_{n-j}|+\cdots+|\ov a_{i+1}|}\nonumber\\
&&\quad\quad
(\ov m_{i+j+1}(\ov a_{i},\cdots,\ov a_1,\ov a_{n+1},\ov a_n,
\cdots,\ov a_{n-j+2},\ov a_{n-j+1}),\ov a_{n-j},\cdots,\ov a_{i+1})\nonumber\\
&+&(-1)^{|\ov a_1|(\sum_{l=1}^{n+1}|\ov a_l|-|\ov a_1|)
+\cdots+|\ov a_{i+1}|(\sum_{l=1}^{n+1}|\ov a_l|-|\ov a_{i+1}|)}(-1)^{|\ov a_{n-j}|+\cdots+|\ov a_{i+2}|}\nonumber\\
&&\quad\quad
(\ov a_{i+1},\ov m_{i+j+1}(\ov a_{i},\cdots,\ov a_1,\ov a_{n+1},\ov a_n,\cdots,
\ov a_{n-j+2},\ov a_{n-j+1}),\ov a_{n-j},\cdots,\ov a_{i+2})\nonumber\\
&+&\cdots\nonumber\\
&+&(-1)^{|\ov a_1|(\sum_{l=1}^{n+1}|\ov a_l|-|\ov a_1|)+\cdots
+|\ov a_{n-j-1}|(\sum_{l=1}^{n+1}|\ov a_l|-|\ov a_{n-j-1}|)}(-1)^{|\ov a_{n-j}|}\nonumber\\
&&\quad\quad
(\ov a_{n-j-1},\cdots,\ov a_{i+1},\ov m_{i+j+1}(\ov a_{i},\cdots,
\ov a_1,\ov a_{n+1},\ov a_n,\cdots,\ov a_{n-j+2},\ov a_{n-j+1}),\ov a_{n-j})\nonumber\\
&+&(-1)^{|\ov a_1|(\sum_{l=1}^{n+1}|\ov a_l|-|\ov a_1|)+\cdots+|\ov a_{n-j}|(\sum_{l=1}^{n+1}|\ov a_l|-|\ov a_{n-j}|)}\nonumber\\
&&\quad\quad
(\ov a_{n-j},\cdots,\ov a_{i+1},\ov m_{i+j+1}(\ov a_{i},\cdots,\ov a_1,\ov a_{n+1},\ov a_n,\cdots,\ov a_{n-j+2},\ov a_{n-j+1})).
\end{eqnarray}
On the other hand, in $Nb=N(b'+b'')$, $Nb'$ does not contribute to this type of terms.
$b''(a_{n+1},\cdots, a_1)$ has one term
$$
(-1)^{\nu_{ij}}(\ov m_{i+j+1}(\ov a_i,\cdots,\ov a_1,\ov a_{n+1},\cdots,\ov a_{n-j+1}),\ov a_{n-j},\cdots,\ov a_{i+1}),
$$
where
\begin{eqnarray*}
\nu_{ij}&=&(|\ov a_1|+\cdots+|\ov a_i|)(|\ov a_{i+1}|+\cdots+|\ov a_{n+1}|)+|\ov a_{n-j}|+\cdots+|\ov a_{i+1}|\\
&\equiv&|\ov a_1|(\sum_{l=1}^{n+1}|\ov a_l|-|\ov a_1|)+\cdots+|\ov a_i|(\sum_{l=1}^{n+1}|\ov a_l|-|\ov a_i|)
+|\ov a_{n-j}|+\cdots+|\ov a_{i+1}|\;({\textrm{mod\,\,2}}),
\end{eqnarray*}
again by the fact that $(\sum_{l=1}^i |\ov a_{l}|)^2-\sum_{l=1}^i |\ov a_{l}|^2$ is even.
Now applying $N$ to $b''(\ov a_{n+1},\cdots, \ov a_1)$ gives
\begin{eqnarray}\label{2nd_right}
&&(-1)^{|\ov a_1|(\sum_{l=1}^{n+1}|\ov a_l|-|\ov a_1|)+\cdots
+|\ov a_i|(\sum_{l=1}^{n+1}|\ov a_l|-|\ov a_i|)+|\ov a_{n-j}|+\cdots+|\ov a_{i+1}|}\nonumber\\
&&\quad\quad
(\ov m_{i+j+1}(\ov a_i,\cdots,\ov a_1,\ov a_{n+1},\cdots,\ov a_{n-j+1}),\ov a_{n-j},\cdots,\ov a_{i+1})\nonumber\\
&+&(-1)^{|\ov a_1|(\sum_{l=1}^{n+1}|\ov a_l|-|\ov a_1|)
+\cdots+|\ov a_i|(\sum_{l=1}^{n+1}|\ov a_l|-|\ov a_i|)+|\ov a_{n-j}|+\cdots
+|\ov a_{i+1}|+|\ov a_{i+1}|(\sum_{l=1}^{n+1}|\ov a_l|-|\ov a_{i+1}|+1)}\nonumber\\
&&\quad\quad
(\ov a_{i+1},\ov m_{i+j+1}(\ov a_i,\cdots,\ov a_1,\ov a_{n+1},\cdots,\ov a_{n-j+1}),\ov a_{n-j},\cdots,\ov a_{i+2})\nonumber\\
&+&\cdots\nonumber\\
&+&(-1)^{|\ov a_1|(\sum_{l=1}^{n+1}|\ov a_l|-|\ov a_1|)
+\cdots+|\ov a_i|(\sum_{l=1}^{n+1}|\ov a_l|-|\ov a_i|)+|\ov a_{n-j}|+\cdots+|\ov a_{i+1}|}\nonumber\\
&&\quad\quad
(-1)^{|\ov a_{i+1}|(\sum_{l=1}^{n+1}|\ov a_l|-|\ov a_{i+1}|+1)+\cdots+|\ov a_{n-j}|
(\sum_{l=1}^{n+1}|\ov a_l|-|\ov a_{n-j}|+1)}\nonumber\\
&&\quad\quad
(\ov a_{n-j},\cdots,\ov a_{i+1},\ov m_{i+j+1}(\ov a_i,\cdots,\ov a_1,\ov a_{n+1},\cdots,\ov a_{n-j+1})).
\end{eqnarray}
One sees \eqref{2nd_left}=\eqref{2nd_right} term by term,
which means $Nb''=\mbox{the second type of summands in}\; b'N$.

Combining the above two cases, we have $b'N=N(b'+b'')=Nb$, which completes the proof.
\end{proof}

\begin{proposition}$b(1-T)=(1-T)b'$.
\end{proposition}

\begin{proof}
Again, for $(\ov a_{n+1},\cdots, \ov a_1)$,
\begin{subequations}
\begin{eqnarray}
&&b(1-T)(\ov a_{n+1},\cdots,\ov a_1)\nonumber\\
&=&(b'+b'')((\ov a_{n+1},\cdots,\ov a_1)-(-1)^{|\ov a_1|(\sum_{l=1}^{n+1}|\ov a_l|-|\ov a_1|)}
(\ov a_1,\ov a_{n+1},\cdots,\ov a_2))\nonumber\\
&=&b'((\ov a_{n+1},\cdots,\ov a_1))+b''((\ov a_{n+1},\cdots,\ov a_1))-
(-1)^{|\ov a_1|(\sum_{l=1}^{n+1}|\ov a_l|-|\ov a_1|)}b'((\ov a_1,\ov a_{n+1},\cdots,\ov a_2))\nonumber\\
&&-(-1)^{|\ov a_1|(\sum_{l=1}^{n+1}|\ov a_l|-|\ov a_1|)}b''((\ov a_1,\ov a_{n+1},\cdots,\ov a_2))\nonumber\\
&=&
\sum_{k=1}^{n+1}\sum_{i=1}^{n-k+2}(-1)^{\varepsilon_k}(\ov a_{n+1},\cdots,\ov a_{k+i},
\ov m_i(\ov a_{k+i-1},\cdots,\ov a_k),\ov a_{k-1},\cdots,\ov a_1)\;\;\;\;\;\label{1a}\\
&+&\sum_{j=0}^{n-1}\sum_{i=1}^{n-j}(-1)^{\nu_{ij}}(\ov m_{i+j+1}(\ov a_i,\cdots,\ov a_1,\ov a_{n+1},
\cdots,\ov a_{n-j+1}),\ov a_{n-j},\cdots, \ov a_{i+1})\nonumber\\
&-&\sum_{k=2}^{n+1}\sum_{i=1}^{n-k+2}(-1)^{|\ov a_1|(\sum_{l=1}^{n+1}|\ov a_l|-|\ov a_1|)}
(-1)^{\varepsilon_k-|\ov a_1|}(\ov a_1,\ov a_{n+1},
\cdots,\ov m_i(\ov a_{k+i-1},\cdots,\ov a_k),\ov a_{k-1},\cdots,\ov a_2)\;\;\;\;\;\;\label{1b}\\
&-&(-1)^{|\ov a_1|(\sum_{l=1}^{n+1}|\ov a_l|-|\ov a_1|)}\ov m_{n+1}(\ov a_1,\ov a_{n+1},\cdots,\ov a_2)\nonumber\\
&-&(-1)^{|\ov a_1|(\sum_{l=1}^{n+1}|\ov a_l|-|\ov a_1|)+|\ov a_2|}
(\ov m_{n}(\ov a_1,\ov a_{n+1},\cdots,\ov a_3),\ov a_2)\nonumber\\
&-&\cdots\nonumber\\
&-&(-1)^{|\ov a_1|(\sum_{l=1}^{n+1}|\ov a_l|-|\ov a_1|)+|\ov a_2+\cdots+|\ov a_n|}
(\ov m_{2}(\ov a_1,\ov a_{n+1}),\ov a_n,\cdots,\ov a_2)\nonumber\\
&-&(-1)^{|\ov a_1|(\sum_{l=1}^{n+1}|\ov a_l|-|\ov a_1|)
+|\ov a_2|+\cdots+|\ov a_{n+1}|}(\ov m_{1}(\ov a_1),\ov a_{n+1},\cdots,\ov a_2)\label{1c}\\
&-&\sum_{j=0}^{n-2}\sum_{i=2}^{n-j}(-1)^{|\ov a_1|(\sum_{l=1}^{n+1}|\ov a_l|-|\ov a_1|)}
(-1)^{(|\ov a_2|+\cdots+|\ov a_i|)(|\ov a_{i+1}|+\cdots+|\ov a_{n+1}|+|\ov a_1|)+|\ov a_{n-j}|+\cdots+|\ov a_{i+1}|}\nonumber\\
&&\quad\quad (\ov m_{i+j+1}(\ov a_i,\cdots,\ov a_2,\ov a_1,\ov a_{n+1},\cdots,\ov a_{n-j+1}),\ov a_{n-j},\cdots,\ov a_{i+1})\nonumber\\
&-&(-1)^{|\ov a_1|(\sum_{l=1}^{n+1}|\ov a_l|-|\ov a_1|)+|\ov a_2|(\sum_{l=1}^{n+1}|\ov a_l|-|\ov a_2|)
+|\ov a_3|+\cdots+|\ov a_{n+1}|}(\ov m_2(\ov a_2,\ov a_1),\ov a_{n+1},\cdots,\ov a_3)\label{1d}\\
&-&\cdots\nonumber\\
&-&(-1)^{|\ov a_1|(\sum_{l=1}^{n+1}|\ov a_l|-|\ov a_1|)+(|\ov a_2|+\cdots+|\ov a_n|)(|\ov a_{n+1}|+|\ov a_1|)
+|\ov a_{n+1}|}(\ov m_n(\ov a_n,\cdots,\ov a_1),\ov a_{n+1})\label{1e}\\
&-&(-1)^{|\ov a_1|(\sum_{l=1}^{n+1}|\ov a_l|-|\ov a_1|)+(|\ov a_2|+\cdots
+|\ov a_{n+1}|)(|\ov a_1|)}(\ov m_{n+1}(\ov a_{n+1,}\cdots,\ov a_1).\label{1f}
\end{eqnarray}
\end{subequations}
Note that, the terms in $b(1-T)$ without labels cancel each other. Now,
\begin{subequations}
\begin{eqnarray}
&& (1-T)b'(\ov a_{n+1},\ov a_n,\cdots,\ov a_1)\nonumber\\
&=& (1-T)(\sum_{k=1}^{n+1}\sum_{i=1}^{n-k+2}
(-1)^{\varepsilon_k}(\ov a_{n+1},\cdots,\ov a_{k+i},\ov m_i(\ov a_{k+i-1},\cdots,\ov a_k),\ov a_{k-1},\cdots,\ov a_1))\nonumber\\
&=& (1-t_n-\cdots t_1-t_0)(\sum_{k=1}^{n+1}\sum_{i=1}^{n-k+2}
(-1)^{\varepsilon_k}(\ov a_{n+1},\cdots,\ov a_{k+i},\ov m_i(\ov a_{k+i-1},\cdots,\ov a_k),\ov a_{k-1},\cdots,\ov a_1))\nonumber\\
&=& \sum_{k=1}^{n+1}\sum_{i=1}^{n-k+2}
(-1)^{\varepsilon_k}(\ov a_{n+1},\cdots,\ov a_{k+i},\ov m_i(\ov a_{k+i-1},\cdots,\ov a_k),\ov a_{k-1},\cdots,\ov a_1)\label{2a}\\
&-&\sum_{k=2}^{n+1}(-1)^{\varepsilon_k}
(-1)^{|\ov a_1|(\sum_{l=1}^{n+1}|\ov a_l|-|\ov a_1|)
-|\ov a_1|}(\ov a_1,\ov a_{n+1},\cdots,\ov a_{k+1},\ov m_1(\ov a_k),\ov a_{k-1},\cdots,\ov a_2)\label{2b1}\\
&-&(-1)^{|\ov a_1|(\sum_{l=1}^{n+1}|\ov a_l|-|\ov a_1|)
+|\ov a_2|+\cdots+|\ov a_{n+1}|}(\ov m_1(\ov a_1),\ov a_{n+1},\cdots,\ov a_2)\label{2c}\\
&-&\sum_{k=2}^n(-1)^{\varepsilon_k}(-1)^{|\ov a_1|(\sum_{l=1}^{n+1}|\ov a_l|-|\ov a_1|)-|\ov a_1|}
(\ov a_1,\ov a_{n+1},\cdots,\ov a_{k+2},\ov m_2(\ov a_{k+1},\ov a_k),\ov a_{k-1},\cdots,\ov a_2)\label{2b2}\\
&-&(-1)^{(|\ov a_2|+|\ov a_1|)(|\ov a_3|+\cdots+|\ov a_{n+1}|)
+|\ov a_3|+\cdots+|\ov a_{n+1}|}(\ov m_2(\ov a_2,\ov a_1),\ov a_{n+1},\cdots,\ov a_3)\label{2d}\\
&-&(-1)^{\varepsilon_2}(-1)^{|\ov a_1|(|\ov a_2|+\cdots
+|\ov a_{n+1}|)+|\ov a_1|}(\ov a_1,\ov m_n(\ov a_{n+1},\cdots,\ov a_2))\label{2b3}\\
&-&(-1)^{(|\ov a_n|+\cdots+|\ov a_1|)|\ov a_{n+1}|+|\ov a_{n+1}|}(\ov m_n(\ov a_n,\cdots,\ov a_1),\ov a_{n+1})\label{2e}\\
&-& \ov m_{n+1}(\ov a_{n+1},\cdots,\ov a_1)\label{2f}.
\end{eqnarray}
\end{subequations}
Now, one can find that  (\ref{1a})=(\ref{2a}),
(\ref{1b})=(\ref{2b1})+(\ref{2b2})+(\ref{2b3}), (\ref{1c})=(\ref{2c}), (\ref{1d})=(\ref{2d}),
(\ref{1e})=(\ref{2e}), and (\ref{1f})=(\ref{2f}), which means
$b(1-T)=(1-T)b'$. This completes the proof.
\end{proof}

\end{appendix}



\begin{thebibliography}{100}
\addtolength{\itemsep}{-1.2ex}

\bibitem{Ab09}M. Abouzaid, {\it A topological model for the Fukaya categories of plumbings},
J. Differential Geom. {\bf 87} (2011), no. 1, 1--80.


\bibitem{BCER}Y. Berest, X. Chen, F. Eshmatov and A. Ramadoss,
 \textit{Noncommutative Poisson structures, derived representation schemes and Calabi-Yau algebras},
{Contemp. Math.} {\bf 583} (2012), 219--246.

\bibitem{BKR}Y. Berest, G. Khachatryana and A. Ramadoss,
\textit{Derived representation schemes and cyclic homology},
Adv. Math. {\bf 245} (2013) 625--689.

\bibitem{BL} R. Bocklandt and L. Le Bruyn, \textit{Necklace Lie algebras and noncommutative symplectic geometry},
Math. Z. \textbf{240} (2002), 141--167.

\bibitem{BS}T. Bridgeland and D. Stern, \textit{Helices on del Pezzo surfaces and tilting Calabi-Yau algebras},
Adv. Math. {\bf 224} (2010) 1672-1716.

\bibitem{CS1}M. Chas and D. Sullivan, {\it String topology}, arXiv:math/9911159.

\bibitem{CS2}M. Chas and D. Sullivan, {\it Closed string operators in topology leading to Lie bialgebras
and higher string algebra}, in {\it The legacy of Niels Henrik Abel}, 771--784, Springer, Berlin, 2004.



\bibitem{Cos07}K. Costello, \textit{Topological conformal field theories and Calabi-Yau categories},
Adv. Math. \textbf{210} (2007), no. 1, 165--214.



\bibitem{CB}W. Crawley-Boevey, \textit{Poisson structures on moduli spaces of
representations}, J. Algebra \textbf{325} (2011), 205--215.



\bibitem{CBEG}W. Crawley-Boevey, P. Etingof and V. Ginzburg,
\textit{Noncommutative geometry and quiver algebras}, Adv. Math. \textbf{209} (2007), 274--336.

\bibitem{Hasegawa}K. Lef\`evre-Hasegawa, \textit{Sur les A$_\infty$-cat\'egories.} Available at:
http://webusers.imj-prg.fr/$\sim$bernhard.keller/lefevre/TheseFinale/tel-00007761.pdf.

\bibitem{Fuk93}K. Fukaya, \textit{Morse homotopy, A$_\infty$-category, and Floer homologies}.
Proceedings of GARC Workshop on Geometry and Topology '93 (Seoul, 1993),
H. J. Kim, ed., Lecture Notes, no. {\bf 18}, Seoul Nat. Univ., Seoul, 1993, 1--102.

\bibitem{Fukaya}K. Fukaya, \textit{Floer homology and mirror symmetry}. II.
Minimal surfaces, geometric analysis and symplectic geometry
(Baltimore, MD, 1999), 31--127, Adv. Stud. Pure Math. {\bf 34}, Math. Soc. Japan, Tokyo, 2002.

\bibitem{Fuk10}K. Fukaya, \textit{Cyclic symmetry and adic convergence in Lagrangian Floer theory}.
Kyoto J. Math. Volume \textbf{50}, Number 3 (2010), 521--590.

\bibitem{FOOO}K. Fukaya, Y.-G. Oh, H. Ohta and K. Ono,
\textit{Lagrangian intersection Floer theory: anomaly and obstruction}. Part I and II.
AMS/IP Studies in Advanced Mathematics {\bf 46}, 2009.

\bibitem{FSS}K. Fukaya, P. Seidel and I. Smith,
\textit{Exact Lagrangian submanifolds in simply-connected cotangent bundles}. Invent. Math. \textbf{172} (2008), no. 1, 1--27.


\bibitem{FSS2}K. Fukaya, P. Seidel and I. Smith, {\it The symplectic geometry of
cotangent bundles from a categorical viewpoint}.
Homological mirror symmetry, 1--26, Lecture Notes in Phys. \textbf{757}, Springer, Berlin, 2009.

\bibitem{GJ}E. Getzler and J.D.S. Jones,
\textit{A$_\infty$-algebras and the cyclic bar complex}, Illinois J. Math. \textbf{34} (1989), 256--283.


\bibitem{Gin} V. Ginzburg, \textit{Noncommutative symplectic geometry,
quiver varieties and operads}, Math. Res. Lett. \textbf{8} (2001), 377--400.

\bibitem{Goldman}W. Goldman,
\textit{Invariant functions on Lie groups and Hamiltonian flows of surface group representations.}
Invent. Math. {\bf 85} (1986), no. 2, 263--302.

\bibitem{Jones}J.D.S. Jones, \textit{Cyclic homology and equivariant homology},
Invent. Math. \textbf{87} (1987), 403--423.


\bibitem{Keller2}B. Keller,
\textit{Invariance and localization for cyclic homology of DG algebras}.
J. Pure Appl. Algebra {\bf 123} (1998), no. 1-3, 223--273.

\bibitem{Keller}B. Keller, \textit{On the cyclic homology of exact categories},
J. Pure Appl. Algebra  \textbf{136} (1999), 1--56.

\bibitem{Ko1} M. Kontsevich, \textit{Feynman diagrams and low-dimensional topology}, First European Congress of
Mathematics, 1992, Paris, Vol. II, Prog. in Math. \textbf{120}, Birkh\"auser 1994, 97--121.

\bibitem{Ko2}M. Kontsevich, \textit{Formal (non)commutative symplectic geometry}, The Gelfand
Mathematical Seminars 1990-1992, Birkh\"auser, Boston, 1993, 173--187.

\bibitem{HMS}M. Kontsevich, \textit{Homological algebra of Mirror Symmetry},
Proceedings of the International Congress of Mathematicians, Z\"urich 1994, vol. I, Birkh\"auser 1995, 120--139.

\bibitem{KR}M. Kontsevich and A. Rosenberg, \textit{Noncommutative smooth spaces}, The Gelfand
Mathematical Seminars 1996-1999, Birkh\"auser, Boston, 2000, 85--108.

\bibitem{KS06}M. Kontsevich and Y. Soibelman, \textit{Notes on A$_\infty$ algebras, A$_\infty$ categories and
non-commutative geometry}. \textit{Homological mirror symmetry},
153--219, Lecture Notes in Phys. {\bf 757}, Springer, Berlin, 2009.


\bibitem{Loday}J.-L. Loday, \textit{Cyclic homology}. Second edition.
Grundlehren der Mathematischen Wissenschaften,
\textbf{301}. Springer--Verlag, Berlin, 1998.



\bibitem{Nadler}D. Nadler, {\it Microlocal branes are constructible sheaves}. Selecta Math. (N.S.) \textbf{15}
(2009), no. 4, 563--619.

\bibitem{PS}M. Penkava, A. Schwarz,
\textit{$A_\infty$-Algebras and the Cohomology of Moduli Spaces.}
Lie groups and Lie algebras: E. B. Dynkin's Seminar, 91--107,
Amer. Math. Soc. Transl. Ser. 2, 169, Amer. Math. Soc., Providence, RI, 1995.


\bibitem{Quillen89}D. Quillen, \textit{Algebra cochains and cyclic cohomology},
Inst. Hautes Etudes Sci. Publ. Math. \textbf{68} (1989), 139--174.

\bibitem{Seidel}P. Seidel, \textit{Fukaya categories and Picard-Lefschetz theory},
Z\"urich Lectures in Advanced Mathematics. European Mathematical Society (EMS), Z\"urich, 2008.

\bibitem{Seidel2}P. Seidel, \textit{Suspending Lefschetz Fibrations, with an Application to Local Mirror Symmetry},
Commun. Math. Phys. {\bf 297}, 515--528 (2010).


\bibitem{VdB}M. Van den Bergh, \textit{Double Poisson algebras},
Trans. Amer. Math. Soc. \textbf{360} (2008), 5711--5769.


\end{thebibliography}
\end{document}